\documentclass{article}
\usepackage{amsmath,amsthm,geometry,graphicx,slashbox,amssymb,bm,comment,color,soul,lscape,natbib}


\numberwithin{equation}{section}

\theoremstyle{plain}
\newtheorem{theorem}{Theorem}[section]
\newtheorem{corollary}[theorem]{Corollary}
\newtheorem{lemma}[theorem]{Lemma}

\theoremstyle{definition}

\newtheorem{assumption}{Assumption}

\theoremstyle{remark}
\newtheorem{remark}{Remark}

\newcommand{\E}{\mathsf{E}}
\newcommand{\V}{\mathsf{V}}
\definecolor{eGreen}{rgb}{.057, .549,.065}

\makeatletter

\@addtoreset{equation}{section}
\makeatother

\begin{document}
\title{Estimators for multivariate allometric regression model}

\setcounter{footnote}{1}
\author{Koji Tsukuda\footnote{Faculty of Mathematics, Kyushu University, 744 Motooka, Nishi-ku, Fukuoka-shi, Fukuoka 819-0395, Japan.} and Shun Matsuura\footnote{Faculty of Science and Technology, Keio University, 3-14-1 Hiyoshi, Kohoku-ku, Yokohama, Kanagawa 223-8522, Japan.}}
\maketitle

\noindent
{\bf Abstract}\\
In a regression model with multiple response variables and multiple explanatory variables, 
if the difference of the mean vectors of the response variables for different values of explanatory variables is always in the direction of the first principal eigenvector of the covariance matrix of the response variables, then it is called a multivariate allometric regression model. 
This paper studies the estimation of the first principal eigenvector in the multivariate allometric regression model. 
A class of estimators that includes conventional estimators is proposed based on weighted sum-of-squares matrices of regression sum-of-squares matrix and residual sum-of-squares matrix. 
We establish an upper bound of the mean squared error of the estimators contained in this class, and the weight value minimizing the upper bound is derived. 
Sufficient conditions for the consistency of the estimators are discussed in weak identifiability regimes under which the difference of the largest and second largest eigenvalues of the covariance matrix decays asymptotically and in ``large $p$, large $n$" regimes, where $p$ is the number of response variables and $n$ is the sample size. 
Several numerical results are also presented.
\vspace{3mm}\\
{\bf Keywords}: 
covariance matrix, 
multivariate regression model, 
principal component analysis,
sum-of-squares matrix,
weak identifiability.
\vspace{3mm}\\
{\bf Mathematics Subject Classification}: 62J05, 62H12, 62H25

\allowdisplaybreaks[4]

\section{Introduction}\label{sec:1}

The multivariate regression model, a regression model with multiple response variables and multiple explanatory variables, is one of the most popular models in multivariate analysis; see \cite{RefA03}, \cite{RefJW07}, \cite{RefRC12}, and so on.
For the model, the method of least squares works well.
When additional assumptions such as low rankness of the regression coefficient matrix exist, there is another estimator that is generally better than the ordinary least squares estimator.
In particular, we consider the multivariate allometric regression model, proposed in \cite{RefTI06}, which is a reduced rank regression model of rank one with the constraint that the differences of the mean vectors of response variables for different values of explanatory variables are in the direction of the first principal eigenvector of the covariance matrix.
The first principal eigenvector is a key quantity to infer the multivariate allometric regression model, so we consider the estimation of this quantity.
Specifically, we define a class of estimators that includes natural conventional estimators, evaluate the mean squared errors of the estimators contained in this class, compare them from the viewpoint of the consistency, propose a new estimator, and examine the proposed estimator through numerical simulations.

Let
\begin{equation}\label{MRA}
\boldsymbol{Y} = \boldsymbol{1}_n\boldsymbol{\mu}' + \boldsymbol{X}\boldsymbol{B} + \boldsymbol{E} 
\ \mbox{with} \ \E[{\rm vec}(\boldsymbol{E})]=\boldsymbol{0}_{np}, \ \V [{\rm vec}(\boldsymbol{E})]=\boldsymbol{\Sigma} \otimes \boldsymbol{I}_n
\end{equation}
be a multivariate regression model, where $\boldsymbol{Y}$ is an $n\times p$ matrix of response variables ($n$ is the sample size and $p$ is the number of response variables), $\boldsymbol{1}_n$ is the $n\times 1$ vector whose all elements are $1$, $\boldsymbol{\mu}$ is a $p\times 1$ vector of intercepts, $\boldsymbol{X}$ is an $n \times q$ matrix of explanatory variables ($q$ is the number of explanatory variables) such that each explanatory variable is 
standardized to have mean $0$ (i.e., $\boldsymbol{X}'\boldsymbol{1}_n = \boldsymbol{0}_q$), $\boldsymbol{B}$ is a $q\times p$ matrix of regression coefficients, $\boldsymbol{E}$ is an $n \times p$ matrix of errors with mean $0$, $\boldsymbol{\Sigma}$ is a $p \times p$ positive-definite matrix denoting the common covariance matrix of errors of the $p$ response variables, $\boldsymbol{0}_{np}$ is the $np \times 1$ zero vector, $\boldsymbol{I}_n$ is the $n\times n$ identity matrix, and $\otimes$ denotes the Kronecker product. 
Putting $\boldsymbol{Y}=\left( \boldsymbol{y}_1,\boldsymbol{y}_2,\dots,\boldsymbol{y}_n \right)'$, $\boldsymbol{X}=\left( \boldsymbol{x}_1,\boldsymbol{x}_2,\dots,\boldsymbol{x}_n \right)'$, and $\boldsymbol{E}=\left( \boldsymbol{e}_1,\boldsymbol{e}_2,\dots,\boldsymbol{e}_n \right)'$, we can rewrite \eqref{MRA} as
\[ \boldsymbol{y}_i = \boldsymbol{\mu} + \boldsymbol{B}'\boldsymbol{x}_i + \boldsymbol{e}_i \ \mbox{with} \ \E[\boldsymbol{e}_i]=\boldsymbol{0}_{p}, \ \V [\boldsymbol{e}_i]=\boldsymbol{\Sigma}, \  i=1,2,\dots,n. \]
Each response vector $\boldsymbol{y}_i$ has the mean vector $\boldsymbol{\mu} + \boldsymbol{B}'\boldsymbol{x}_i$ and the covariance matrix $\boldsymbol{\Sigma}$. 
The multivariate allometric regression model is a special case of the multivariate regression model.
Let $\lambda_1,\lambda_2,\dots,\lambda_p$ denote the ordered eigenvalues of $\boldsymbol{\Sigma}$ 
($\lambda_1 \ge \lambda_2 \ge \cdots \ge \lambda_p > 0$), 
and let $\boldsymbol{\gamma}_i$ be a normalized principal eigenvector of $\boldsymbol{\Sigma}$ corresponding to its eigenvalue $\lambda_i$ for $i = 1,\ldots,p$. 
Suppose that $\lambda_1 > \lambda_2$ holds, which implies that $\boldsymbol{\gamma}_1$ is unique up to a sign change. 
Then, the multivariate allometric regression model is expressed as
\begin{align}
& \boldsymbol{Y} = \boldsymbol{1}_n\boldsymbol{\mu}' + \boldsymbol{X}\boldsymbol{\alpha}\boldsymbol{\gamma}_1' + \boldsymbol{E} \nonumber \\
& \ \mbox{with} \ \E[{\rm vec}(\boldsymbol{E})]=\boldsymbol{0}_{np}, \ \V [{\rm vec}(\boldsymbol{E})]= \boldsymbol{\Sigma}\otimes \boldsymbol{I}_n 
=\boldsymbol{\Gamma}\boldsymbol{\Lambda}\boldsymbol{\Gamma}' \otimes \boldsymbol{I}_n,
\label{MAR}
\end{align}
where $\boldsymbol{\alpha}$ is a $q \times 1$ vector, $\boldsymbol{\Gamma}=\left( \boldsymbol{\gamma}_1,\boldsymbol{\gamma}_2,\dots,\boldsymbol{\gamma}_p \right)$ is a $p\times p$ orthogonal matrix, and $\boldsymbol{\Lambda}={\rm diag}\left( \lambda_1,\lambda_2,\dots,\lambda_p \right)$ is a $p\times p$ diagonal matrix. 
We assume the multivariate normality of the distribution of the error matrix $\boldsymbol{E}$:
${\rm vec}(\boldsymbol{E}) \sim \mathcal{N}_{np} \left( \boldsymbol{0}_{np}, \boldsymbol{\Gamma}\boldsymbol{\Lambda}\boldsymbol{\Gamma}' \otimes \boldsymbol{I}_n \right)$.
In this setup, we consider the estimation of $\boldsymbol{\gamma}_1$ in \eqref{MAR}.

The multivariate allometric regression model was proposed as an extension of the allometric extension model \citep{RefF97,RefH06} to the context of multivariate regression in \cite{RefTI06}.
The allometric extension model means that several population distributions share a first principal eigenvector of their covariance matrices and the differences of their mean vectors are in the direction of the first principal eigenvector, which has been applied to various areas including biometrics \citep{RefKF91,RefF97}.
\cite{RefBFN99} proposed a hypothesis testing procedure for the allometric extension model based on a $\chi^2$ statistic in large sample situations. 
\cite{RefKHF08} derived conditions for multivariate conditional distributions to form the allometric extension model. 
\cite{RefMK14} discussed principal points of multivariate mixture distributions of the allometric extension model.
\cite{RefTM23} proposed hypothesis testing procedures for the allometric extension model in high-dimensional and strongly spiked eigenvalue situations.
\cite{RefKGT23} evaluated the performance of spectral clustering algorithm when it is used for a random sample from the mixture of two subpopulations that are allometric extensions of each other.

The multivariate allometric regression model can be viewed as a partial generalization of the allometric extension model, because the multivariate allometric regression model \eqref{MAR} can deal with infinite mixture of population distributions incorporating multivariate regression models with explanatory variables, while there is an restriction that each population distribution has a common covariance matrix of errors. 
Hypothesis testing for the multivariate allometric regression model has been discussed in \cite{RefS17}.
The multivariate allometric regression model is also closely related to multivariate reduced rank regression \citep{RefA51,RefA03,RefI75,RefRVC22}. 
If the differences of the mean vectors of the response variables for different explanatory variables 
are in a common direction but do not necessarily coincide with the first principal eigenvector $\boldsymbol{\gamma}_1$ of $\boldsymbol{\Sigma}$, then \eqref{MAR} reduces to the multivariate reduced rank regression of rank one. 
However, in \eqref{MAR}, using the information that the common direction of the differences of the mean vectors coincides with the first principal eigenvector may enable more accurate estimation compared with the multivariate reduced rank regression of rank one. 

This paper studies the estimation of $\boldsymbol{\gamma}_1$ in \eqref{MAR}. 
From the viewpoint of principal component analysis, inferring the first principal eigenvector $\boldsymbol{\gamma}_1$ is of special interest.
Moreover, although we focus on the estimation of $\boldsymbol{\gamma}_1$, precise estimation of $\boldsymbol{\gamma}_1$ may lead to precise estimation of $\boldsymbol{B} = \boldsymbol{\alpha}\boldsymbol{\gamma}_1' = \boldsymbol{\alpha}\boldsymbol{\gamma}_1' \boldsymbol{\gamma}_1 \boldsymbol{\gamma}_1'$. 
In fact, once an estimator $\hat{\boldsymbol{\gamma}}_1$ of $\boldsymbol{\gamma}_1$ is obtained, an estimator of $\boldsymbol{B}$ is obtained as $\hat{\boldsymbol{B}}_{OLS} \hat{\boldsymbol{\gamma}}_1 \hat{\boldsymbol{\gamma}}_1'$ if $\boldsymbol{X}' \boldsymbol{X}$ is invertible, where $\hat{\boldsymbol{B}}_{OLS}= \left( \boldsymbol{X}'\boldsymbol{X} \right)^{-1}\boldsymbol{X}'\boldsymbol{Y}$ is the ordinary least squares estimator. 
In the reduced rank regression, such estimators are generally better than $\hat{\boldsymbol{B}}_{OLS}$; see \cite{RefA99}.
Also when \eqref{MAR} is supposed, it seems natural to consider that $\hat{\boldsymbol{B}}_{OLS} \hat{\boldsymbol{\gamma}}_1 \hat{\boldsymbol{\gamma}}_1'$ is better than $\hat{\boldsymbol{B}}_{OLS}$ in general.
This is another motivation to focus on estimating $\boldsymbol{\gamma}_1$.
In particular, the consistency up to sign of $\hat{\boldsymbol{\gamma}}_1$ is important.
That is because, e.g., under some adequate asymptotic regimes, if $\hat{\boldsymbol{\gamma}}_1$ is consistent, then $\hat{\boldsymbol{B}}_{OLS} \hat{\boldsymbol{\gamma}}_1 \hat{\boldsymbol{\gamma}}_1'$ is also consistent.

In this paper, we deal with a class of estimators of $\boldsymbol{\gamma}_1$ based on weighted matrices of two independent sum-of-squares matrices: regression sum-of-squares matrix and residual sum-of-squares matrix. 
Non-asymptotic upper bounds are established for the mean squared errors of the estimators, 
and the weight value minimizing the upper bound for the mean squared error is derived. 
Subsequently, we discuss the consistency of the estimators under 
weak identifiability regimes under which $\lambda_1-\lambda_2 \to 0$ as $n \to \infty$, and 
``large $p$, large $n$'' regimes under which $p$ and $n$ tend to $\infty$ simultaneously.
As remarked in \cite{RefPRV20}, properties of principal component analysis under weak identifiability regimes have not been sufficiently investigated.
Under large $p$, large $n$ regimes that include so-called high-dimensional asymptotic regimes $n,p\to\infty$ with $p/n \to \infty$, we consider weak identifiability regimes and a sort of spiked eigenvalue regimes simultaneously.
In the context of multivariate regression, large $p$, large $n$ regimes have been discussed in a lot of works such as \cite{RefBW19}, \cite{RefBSW11}, \cite{RefCDC13}, \cite{RefG11}, \cite{RefSC17}, \cite{RefYWF15}, \cite{RefYELM07}, and so on.
Furthermore, numerical results are presented to compare the proposed estimator with some conventional estimators. 
Finally, a illustrative real data example is provided.

\section{Conventional estimators}\label{sec:2}

This section presents some conventional estimators of $\boldsymbol{\gamma}_1$ in the multivariate allometric regression model \eqref{MAR}. 
Here and subsequently, assume that $n > 1+q$ and $\mathrm{rank}(\boldsymbol{X}) = q$.
Note that $n>2+q$ will be assumed in Section 6 in order to avoid zero in a denominator.
Let
\begin{align*}
& \boldsymbol{S}_R = \boldsymbol{Y}'\boldsymbol{X}\left( \boldsymbol{X}'\boldsymbol{X} \right)^{-1}\boldsymbol{X}'\boldsymbol{Y}, \\
& \boldsymbol{S}_E = \boldsymbol{Y}' \left\{ \boldsymbol{I}_n - \frac{1}{n}\boldsymbol{1}_n\boldsymbol{1}_n' - \boldsymbol{X}\left( \boldsymbol{X}'\boldsymbol{X} \right)^{-1}\boldsymbol{X}' \right\} \boldsymbol{Y}, \\
& \boldsymbol{S}_T = \boldsymbol{Y}' \left( \boldsymbol{I}_n - \frac{1}{n}\boldsymbol{1}_n\boldsymbol{1}_n' \right) \boldsymbol{Y}.
\end{align*}
The matrices $\boldsymbol{S}_R$, $\boldsymbol{S}_E$, and $\boldsymbol{S}_T$ are so-called regression sum-of-squares, residual sum-of-squares, and total sum-of-squares matrices, respectively.
It is obvious that $\boldsymbol{S}_T=\boldsymbol{S}_R+\boldsymbol{S}_E$.
Straightforward calculations give
\begin{align*}
&\E \left[ \boldsymbol{S}_R \right] = 
q\boldsymbol{\Sigma} + \boldsymbol{\gamma}_1\boldsymbol{\alpha}'\boldsymbol{X}'\boldsymbol{X}\boldsymbol{\alpha}\boldsymbol{\gamma}_1' 
= q\boldsymbol{\Gamma}\boldsymbol{\Lambda}\boldsymbol{\Gamma}' 
+ \| \boldsymbol{X}\boldsymbol{\alpha} \|^2 \boldsymbol{\gamma}_1\boldsymbol{\gamma}_1', \\
&\E \left[ \boldsymbol{S}_E \right] = (n-1-q)\boldsymbol{\Sigma} = (n-1-q)\boldsymbol{\Gamma}\boldsymbol{\Lambda}\boldsymbol{\Gamma}', \\
&\E \left[ \boldsymbol{S}_T \right] 
= (n-1)\boldsymbol{\Sigma} + \boldsymbol{\gamma}_1\boldsymbol{\alpha}'\boldsymbol{X}'\boldsymbol{X}\boldsymbol{\alpha}\boldsymbol{\gamma}_1' 
= (n-1)\boldsymbol{\Gamma}\boldsymbol{\Lambda}\boldsymbol{\Gamma}' 
+ \| \boldsymbol{X}\boldsymbol{\alpha} \|^2 \boldsymbol{\gamma}_1\boldsymbol{\gamma}_1'. 
\end{align*}

Let $\hat{\boldsymbol{\gamma}}_1^{(R)}$, $\hat{\boldsymbol{\gamma}}_1^{(E)}$, and $\hat{\boldsymbol{\gamma}}_1^{(T)}$ be the first principal eigenvectors of $\boldsymbol{S}_R$, $\boldsymbol{S}_E$, and $\boldsymbol{S}_T$, respectively, 
which estimate $\boldsymbol{\gamma}_1$. 
The estimator $\hat{\boldsymbol{\gamma}}_1^{(R)}$ is a typical estimator to estimate the regression coefficient matrix under reduced rank regression of rank one; see \cite{RefRVC22}.
However, in the multivariate allometric regression model \eqref{MAR}, $\hat{\boldsymbol{\gamma}}_1^{(R)}$ does not sufficiently use the information that the first principal eigenvector of $\boldsymbol{\Sigma}$ is $\boldsymbol{\gamma}_1$. 
The estimator $\hat{\boldsymbol{\gamma}}_1^{(E)}$ only uses this information, not using the information that $\boldsymbol{B}$ is decomposed as $\boldsymbol{B}=\boldsymbol{\alpha}\boldsymbol{\gamma}_1'$. 
The estimator $\hat{\boldsymbol{\gamma}}_1^{(T)}$ has been presented in Tarpey and Ivey (2006), which uses both the informations. 

However, it is natural to consider that $\hat{\boldsymbol{\gamma}}^{(T)}$ is not always the best estimator. 
To be more specific, for example, we consider the accuracy of the first principal eigenvector $\hat{\boldsymbol{\gamma}}^{(E)}_1$ of $\boldsymbol{S}_E$. 
As $n\to\infty$ with $q$ held fixed, under the assumption that $n^{-1} \boldsymbol{X}' \boldsymbol{X}$ converges to a positive-definite matrix, $\sqrt{n} (\hat{\boldsymbol{\gamma}}^{(E)}_1 - \boldsymbol{\gamma}_1)$ converges in distribution to
\[
\mathcal{N}_p \left( \boldsymbol{0}_p, \lambda_1 \sum_{i=2}^p \frac{\lambda_i}{(\lambda_1 - \lambda_i)^2} \boldsymbol{\gamma}_i \boldsymbol{\gamma}_i' \right) ;
\]
see, e.g., \cite{RefKN93}. 
This suggests that if $\lambda_1 - \lambda_2$ is small, then $\hat{\boldsymbol{\gamma}}^{(E)}_1$ is not a good estimator. 
Hence, we consider that this property may be inherited by $\hat{\boldsymbol{\gamma}}^{(T)}_1$, since $\boldsymbol{S}_T = \boldsymbol{S}_R + \boldsymbol{S}_E$ includes $\boldsymbol{S}_E$. 
This estimate will be numerically examined in Section~\ref{sec:7}.

\begin{remark}
As stated above, $\mathrm{rank}(\boldsymbol{X}) = q$ is assumed throughout the paper.
When $\mathrm{rank}(\boldsymbol{X}) < q$, we may extend analysis procedures studied in this paper by replacing $(\boldsymbol{X}' \boldsymbol{X})^{-1}$ with the Moore--Penrose inverse of $\boldsymbol{X}' \boldsymbol{X}$ or removing redundant explanatory variables with modifying the corresponding parts such as degrees of freedom.
\end{remark}

\begin{remark}\label{remloc}
Note that $\hat{\boldsymbol{\gamma}}_1^{(T)}$ is the first principal eigenvector in the usual principal component analysis for the multivariate location-scale model with finite second moment that is given by $\boldsymbol{\alpha} = \boldsymbol{0}_q$ in \eqref{MAR}.
\end{remark}

\begin{remark}
\cite{RefPRV20} studied consistency and asymptotic normality of sample first principal eigenvector under weak identifiability regimes with a specific form of the covariance matrix.
Moreover, considering elliptical distributions, \cite{RefPV23} studied asymptotic properties of the leading eigenvector of the Tyler shape estimator.
\end{remark}

\begin{remark}
In our setup, $\hat{\boldsymbol{\gamma}}_1^{(E)}$ is the first principal eigenvector of a Wishart matrix with $n - 1 - q$ degrees of freedom and scale matrix parameter $\boldsymbol{\Sigma}$.
Thus, if $q$ is fixed, properties of $\hat{\boldsymbol{\gamma}}_1^{(E)}$ can be interpreted as general properties for the Wishart distribution.
\end{remark}

\section{Estimators based on weighted matrices of $\boldsymbol{S}_R$ and $\boldsymbol{S}_E$}\label{sec:3}

Let us introduce a class of estimators of $\boldsymbol{\gamma}_1$ based on weighted matrices of $\boldsymbol{S}_R$ and $\boldsymbol{S}_E$, 
and discuss some properties of the estimators such as their mean squared errors and consistency. 
Let $\hat{\boldsymbol{\gamma}}_1(w)$ be 
the first principal eigenvector of a weighted matrix of $\boldsymbol{S}_R$ and $\boldsymbol{S}_E$: 
\[ \boldsymbol{S}(w) = (1-w)\boldsymbol{S}_R + w\boldsymbol{S}_E \]
with $0\le w \le 1$. 
Obviously, $\hat{\boldsymbol{\gamma}}_1(0)$, $\hat{\boldsymbol{\gamma}}_1(0.5)$, and $\hat{\boldsymbol{\gamma}}_1(1)$ 
correspond to $\hat{\boldsymbol{\gamma}}_1^{(R)}$, $\hat{\boldsymbol{\gamma}}_1^{(T)}$, and $\hat{\boldsymbol{\gamma}}_1^{(E)}$, respectively. 
In this section, we establish an upper bound for the mean squared error of the estimator $\hat{\boldsymbol{\gamma}}_1(w)$. 

As a preparation, we derive a simple form of the expectation of the squared Frobenius norm of $\boldsymbol{S}(w) - \E[\boldsymbol{S}(w)]$.

\begin{lemma}\label{Lem1}
It holds that
\begin{align*}
&\E \left[ \left\| \boldsymbol{S}(w) - \E[\boldsymbol{S}(w)] \right\|_{\rm F}^{2} \right] \\
& =
 \{ q(1 - 2w) + (n-1)w^2 \} \{ {\rm tr}( \boldsymbol{\Sigma}^2 ) + ({\rm tr}(\boldsymbol{\Sigma}))^2 \} \\
&\quad 
+ 2(1-w)^2 (\lambda_1 + {\rm tr}(\boldsymbol{\Sigma})) \| \boldsymbol{X} \boldsymbol{\alpha} \|^2.
\end{align*}
\end{lemma}

\begin{proof}
From the assumption of 
${\rm vec}(\boldsymbol{E}) \sim \mathcal{N}_{np} \left( \boldsymbol{0}_{np}, \boldsymbol{\Sigma} \otimes \boldsymbol{I}_n \right)$, 
it follows that $\boldsymbol{S}_R$ and $\boldsymbol{S}_E$ are independent each other; see, e.g., \cite{RefA03}.
On account of this independence, we have
\begin{align*}
& \E \left[ \left\| \boldsymbol{S}(w) - \E[\boldsymbol{S}(w)] \right\|_{\rm F}^{2} \right] \\
&= \E \left[ {\rm tr} \left( \left( \boldsymbol{S}(w) - \E[\boldsymbol{S}(w)] \right)'\left( \boldsymbol{S}(w) - \E[\boldsymbol{S}(w)] \right) \right) \right] \\
&= \E \left[ {\rm tr} \left( \left( \boldsymbol{S}(w) - \E[\boldsymbol{S}(w)] \right)^2 \right) \right] \\
&= (1-w)^2 \E \left[ {\rm tr} \left( \left( \boldsymbol{S}_R - \E[\boldsymbol{S}_R] \right)^2 \right) \right]
+ w^2 \E \left[ {\rm tr} \left( \left( \boldsymbol{S}_E - \E[\boldsymbol{S}_E] \right)^2 \right) \right].
\end{align*}

First, we evaluate 
$\E [ {\rm tr} ( ( \boldsymbol{S}_R - \E[\boldsymbol{S}_R] )^2 ) ]$
as
\begin{align*}
& \E \Bigl[ {\rm tr} \Bigl( \bigl( \boldsymbol{S}_R - \E[\boldsymbol{S}_R] \bigr)^2 \Bigr) \Bigr]  \\
&= \E \Bigl[ {\rm tr} \Bigl( \bigl\{ 
\boldsymbol{Y}' \boldsymbol{X}\left( \boldsymbol{X}'\boldsymbol{X} \right)^{-1}\boldsymbol{X}' \boldsymbol{Y} 
- \bigl( q\boldsymbol{\Sigma} + \boldsymbol{\gamma}_1\boldsymbol{\alpha}'\boldsymbol{X}'\boldsymbol{X}\boldsymbol{\alpha}\boldsymbol{\gamma}_1' \bigr) \bigr\}^2 \Bigr) \Bigr] \\
&= \E \Bigl[ {\rm tr} \Bigl( \Bigl\{ 
\left( \boldsymbol{1}_n\boldsymbol{\mu}' + \boldsymbol{X}\boldsymbol{\alpha}\boldsymbol{\gamma}_1' + \boldsymbol{E}\right)'
\boldsymbol{X}\left( \boldsymbol{X}'\boldsymbol{X} \right)^{-1}\boldsymbol{X}'\left( \boldsymbol{1}_n\boldsymbol{\mu}' + \boldsymbol{X}\boldsymbol{\alpha}\boldsymbol{\gamma}_1' + \boldsymbol{E} \right) 
\\ & \qquad
- \left( q\boldsymbol{\Sigma} + \boldsymbol{\gamma}_1\boldsymbol{\alpha}'\boldsymbol{X}'\boldsymbol{X}\boldsymbol{\alpha}\boldsymbol{\gamma}_1' \right) \Bigr\}^2 \Bigr) \Bigr] \\
&= \E \left[ {\rm tr} \left( 
\left\{ \boldsymbol{E}'\boldsymbol{X}\left( \boldsymbol{X}'\boldsymbol{X} \right)^{-1}\boldsymbol{X}'\boldsymbol{E} 
+ \boldsymbol{E}'\boldsymbol{X}\boldsymbol{\alpha}\boldsymbol{\gamma}_1' 
+ \boldsymbol{\gamma}_1\boldsymbol{\alpha}'\boldsymbol{X}'\boldsymbol{E} 
- q\boldsymbol{\Sigma} \right\}^2 \right) \right].
\end{align*}
We note that $\boldsymbol{E}'\boldsymbol{X}\left( \boldsymbol{X}'\boldsymbol{X} \right)^{-1}\boldsymbol{X}'\boldsymbol{E}$ is distributed as 
the $p$-dimensional Wishart distribution 
$\mathcal{W}_p \left( q, \boldsymbol{\Sigma} \right)$
 with 
$q$ degrees of freedom and scale matrix parameter $\boldsymbol{\Sigma}$. 
It is well known that if $\boldsymbol{W} \sim \mathcal{W}_p \left( q, \boldsymbol{\Sigma} \right)$, then 
\[ \E \left[ \boldsymbol{W} \right] = q\boldsymbol{\Sigma} , \quad
\E \left[ {\rm tr}\left( \boldsymbol{W}^2 \right) \right]
=(q+q^2){\rm tr}\left( \boldsymbol{\Sigma}^2 \right) +q\left( {\rm tr}\left( \boldsymbol{\Sigma} \right) \right)^2. \] 
Moreover, as it will be seen in Lemma~\ref{lem32} below, it also holds that
\[ \E \left[ \boldsymbol{E}'\boldsymbol{X}\left( \boldsymbol{X}'\boldsymbol{X} \right)^{-1}\boldsymbol{X}'\boldsymbol{E}\boldsymbol{E}'\boldsymbol{X} \right] = \boldsymbol{0}_{p\times q},\]
where $\boldsymbol{0}_{p\times q}$ is the $p \times q$ zero matrix.
Hence, letting $\boldsymbol{W} = \boldsymbol{E}'\boldsymbol{X}\left( \boldsymbol{X}'\boldsymbol{X} \right)^{-1}\boldsymbol{X}'\boldsymbol{E}$, we have
\begin{align*}
& \E \left[ {\rm tr} \left( 
\left\{ \boldsymbol{E}'\boldsymbol{X}\left( \boldsymbol{X}'\boldsymbol{X} \right)^{-1}\boldsymbol{X}'\boldsymbol{E} 
+ \boldsymbol{E}'\boldsymbol{X}\boldsymbol{\alpha}\boldsymbol{\gamma}_1' 
+ \boldsymbol{\gamma}_1\boldsymbol{\alpha}'\boldsymbol{X}'\boldsymbol{E} 
- q\boldsymbol{\Sigma} \right\}^2 \right) \right] \\
&= \E \left[ {\rm tr} \left( 
\left( \boldsymbol{W}
+ \boldsymbol{E}'\boldsymbol{X}\boldsymbol{\alpha}\boldsymbol{\gamma}_1' 
+ \boldsymbol{\gamma}_1\boldsymbol{\alpha}'\boldsymbol{X}'\boldsymbol{E} 
- q\boldsymbol{\Sigma} \right)^2 \right) \right] \\
&= \E \Bigl[ {\rm tr} \Bigl( 
\boldsymbol{W}^2 
- q \boldsymbol{W} \boldsymbol{\Sigma} - q\boldsymbol{\Sigma} \boldsymbol{W} 
+ q^2\boldsymbol{\Sigma}^2 
+ \boldsymbol{E}'\boldsymbol{X}\boldsymbol{\alpha}\boldsymbol{\gamma}_1'\boldsymbol{E}'\boldsymbol{X}\boldsymbol{\alpha}\boldsymbol{\gamma}_1'
\\ &\qquad
+ \boldsymbol{\gamma}_1\boldsymbol{\alpha}'\boldsymbol{X}'\boldsymbol{E}\boldsymbol{\gamma}_1\boldsymbol{\alpha}'\boldsymbol{X}'\boldsymbol{E}
+ \boldsymbol{E}'\boldsymbol{X}\boldsymbol{\alpha}\boldsymbol{\gamma}_1'\boldsymbol{\gamma}_1\boldsymbol{\alpha}'\boldsymbol{X}'\boldsymbol{E}
+ \boldsymbol{\gamma}_1\boldsymbol{\alpha}'\boldsymbol{X}'\boldsymbol{E}\boldsymbol{E}'\boldsymbol{X}\boldsymbol{\alpha}\boldsymbol{\gamma}_1'  
\Bigr) \Bigr] \\
&= \E \Bigl[ {\rm tr} \Bigl( 
\boldsymbol{W}^2 - q^2\boldsymbol{\Sigma}^2 \Bigr) 
+ 2\left( \boldsymbol{\alpha}'\boldsymbol{X}'\boldsymbol{E}\boldsymbol{\gamma}_1 \right)^2 
+ 2\boldsymbol{\alpha}'\boldsymbol{X}'\boldsymbol{E}\boldsymbol{E}'\boldsymbol{X}\boldsymbol{\alpha} \Bigr] \\
&= q{\rm tr}( \boldsymbol{\Sigma}^2 ) + q\left( {\rm tr}( \boldsymbol{\Sigma} ) \right)^2 
+ 2\lambda_1 \| \boldsymbol{X}\boldsymbol{\alpha} \|^2
+ 2{\rm tr} ( \boldsymbol{\Sigma} ) \| \boldsymbol{X}\boldsymbol{\alpha} \|^2.
\end{align*}

Next, we evaluate 
$ \E [ {\rm tr} ( ( \boldsymbol{S}_E - \E[\boldsymbol{S}_E] )^2 ) ]$. 
Since $\boldsymbol{S}_E \sim \mathcal{W}_{p}\left( n-1-q, \boldsymbol{\Sigma} \right)$, it holds that
\begin{align*}
& \E \left[ {\rm tr} \left( \left( \boldsymbol{S}_E - \E[\boldsymbol{S}_E] \right)^2 \right) \right] \\
&= \E \left[ {\rm tr} \left( \boldsymbol{S}_{E}^{2} \right) \right] - {\rm tr} \left( \left( \E[\boldsymbol{S}_E] \right)^2 \right) \\
&= \left\{ (n-1-q)+(n-1-q)^2\right\} {\rm tr}\left( \boldsymbol{\Sigma}^2 \right) \\
&\quad + (n-1-q)\left( {\rm tr}\left( \boldsymbol{\Sigma} \right) \right)^2 
- (n-1-q)^2 {\rm tr}\left( \boldsymbol{\Sigma}^2 \right) \\
&= (n-1-q){\rm tr}( \boldsymbol{\Sigma}^2 ) + (n-1-q)\left( {\rm tr}( \boldsymbol{\Sigma} ) \right)^2.
\end{align*}

We see from the above calculations that
\begin{align*}
& \E \left[ \left\| \boldsymbol{S}(w) - \E[\boldsymbol{S}(w)] \right\|_{\rm F}^{2} \right] \\
&= (1-w)^2 \E \left[ {\rm tr} \left( \left( \boldsymbol{S}_R - \E[\boldsymbol{S}_R] \right)^2 \right) \right]
+ w^2 \E \left[ {\rm tr} \left( \left( \boldsymbol{S}_E - \E[\boldsymbol{S}_E] \right)^2 \right) \right] \\
&= (1-w)^2 \left\{ q{\rm tr}( \boldsymbol{\Sigma}^2 ) + q\left( {\rm tr}( \boldsymbol{\Sigma} ) \right)^2 
+ 2\left( \lambda_1 + {\rm tr}( \boldsymbol{\Sigma} ) \right) \| \boldsymbol{X}\boldsymbol{\alpha} \|^2 \right\} \\
&\quad + w^2\left\{ (n-1-q){\rm tr}( \boldsymbol{\Sigma}^2 ) + (n-1-q)\left( {\rm tr}( \boldsymbol{\Sigma} ) \right)^2 \right\} \\
&= \{ q(1 - 2w) + (n-1)w^2 \} \{ {\rm tr}( \boldsymbol{\Sigma}^2 ) 
+ ({\rm tr}(\boldsymbol{\Sigma}))^2 \} \\
&\quad + 2(1-w)^2 (\lambda_1 + {\rm tr}(\boldsymbol{\Sigma})) \| \boldsymbol{X} \boldsymbol{\alpha} \|^2.
\end{align*}
This completes the proof.
\end{proof}

The following lemma is already used in the proof above.

\begin{lemma}\label{lem32}
It holds that
\[\E \left[ \boldsymbol{E}'\boldsymbol{X}\left( \boldsymbol{X}'\boldsymbol{X} \right)^{-1}\boldsymbol{X}'\boldsymbol{E}\boldsymbol{E}'\boldsymbol{X} \right] = \boldsymbol{0}_{p\times q} . \]
\end{lemma}

\begin{proof}
Let $\boldsymbol{g}_1',\boldsymbol{g}_2',\dots,\boldsymbol{g}_q'$ be the row vectors of $\boldsymbol{G}= \left( \boldsymbol{X}'\boldsymbol{X} \right)^{-1/2}\boldsymbol{X}'\boldsymbol{E}$, that is, $\boldsymbol{G} = \left( \boldsymbol{g}_1,\boldsymbol{g}_2,\dots,\boldsymbol{g}_q \right)'$, where $ ( \boldsymbol{X}'\boldsymbol{X} )^{-1/2}$ is the $q \times q$ symmetric matrix satisfying that 
\[ \{ ( \boldsymbol{X}'\boldsymbol{X} )^{-1/2} \}^2  =  \left( \boldsymbol{X}'\boldsymbol{X} \right)^{-1}. \] 
Then, from $\V [{\rm vec}(\boldsymbol{G})]= \boldsymbol{\Sigma} \otimes \boldsymbol{I}_q$, it follows that $\boldsymbol{g}_1,\boldsymbol{g}_2,\dots,\boldsymbol{g}_q$ are independent and identically distributed as $\mathcal{N}_p(\boldsymbol{0}_p,\boldsymbol{\Sigma})$, and thus we see that
\begin{align*}
& \E \left[ \boldsymbol{E}'\boldsymbol{X}\left( \boldsymbol{X}'\boldsymbol{X} \right)^{-1}\boldsymbol{X}'\boldsymbol{E}\boldsymbol{E}'\boldsymbol{X} \right]  
\\ 
& = \E \left[ \boldsymbol{G}'\boldsymbol{G}\boldsymbol{G}'\left( \boldsymbol{X}'\boldsymbol{X} \right)^{1/2} \right] 
= \E \left[ \sum_{i=1}^{q} \boldsymbol{g}_i \boldsymbol{g}_i' \left( \boldsymbol{g}_1,\boldsymbol{g}_2,\dots,\boldsymbol{g}_q \right) \right] \left( \boldsymbol{X}'\boldsymbol{X} \right)^{1/2} 
= \boldsymbol{0}_{p\times q}.
\end{align*}
This completes the proof.
\end{proof}

Now, we establish an upper bound without any unspecified constant for the mean squared error (up to sign) of the estimator $\hat{\boldsymbol{\gamma}}_1(w)$: 
\[ \E \left[ \min_{\theta = -1,1} \| \theta \hat{\boldsymbol{\gamma}}_1(w) -  \boldsymbol{\gamma}_1 \|^2 \right] \]
in the following theorem.

\begin{theorem}\label{thm1}
It holds that
\begin{align}
& \E \left[ \min_{\theta = -1,1} \| \theta \hat{\boldsymbol{\gamma}}_1(w) -  \boldsymbol{\gamma}_1 \|^2 \right] \nonumber \\
&\le \frac{8 a \{ q(1 - 2w) + (n-1)w^2 \}  + 16 bc (1-w)^2  }
{ \left[d  \{ q + (n-1-2q)w \} + c (1-w) \right]^2 }, \label{UP}
\end{align}
where
\begin{equation}\label{abcd}
a = {\rm tr}( \boldsymbol{\Sigma}^2 ) + ({\rm tr}(\boldsymbol{\Sigma}))^2, \
b = \lambda_1 + {\rm tr}(\boldsymbol{\Sigma}), \
c = \| \boldsymbol{X} \boldsymbol{\alpha} \|^2, \
d = \lambda_1 - \lambda_2.
\end{equation}

\end{theorem}

\begin{proof}
First, let 
$\lambda_1(w)$ and $\lambda_2(w)$ be the largest and the second largest eigenvalues of $\E[\boldsymbol{S}(w)]$, respectively.
It follows from
\begin{align*}
\E[\boldsymbol{S}(w)] 
&= \E \left[ (1-w)\boldsymbol{S}_R + w\boldsymbol{S}_E \right] \\
&= (1-w) \left( q\boldsymbol{\Gamma}\boldsymbol{\Lambda}\boldsymbol{\Gamma}' 
+ \| \boldsymbol{X}\boldsymbol{\alpha} \|^2 \boldsymbol{\gamma}_1\boldsymbol{\gamma}_1' \right)
+ w(n-1-q)\boldsymbol{\Gamma}\boldsymbol{\Lambda}\boldsymbol{\Gamma}'  \\
&=  \left\{ q(1-w) + (n-1-q)w \right\} \boldsymbol{\Gamma}\boldsymbol{\Lambda}\boldsymbol{\Gamma}' 
+ (1-w)\| \boldsymbol{X}\boldsymbol{\alpha} \|^2 \boldsymbol{\gamma}_1\boldsymbol{\gamma}_1' 
\end{align*}
that
\begin{align*}
& \lambda_1(w) = \left\{ q(1-w) + (n-1-q)w \right\} \lambda_1 + (1-w) \| \boldsymbol{X}\boldsymbol{\alpha} \|^2 , \\
& \lambda_2(w) = \left\{ q(1-w) + (n-1-q)w \right\} \lambda_2,
\end{align*}
and that $\boldsymbol{\gamma}_1$ is the principal eigenvector of $\E[\boldsymbol{S}(w)] $ corresponding to $\lambda_1(w)$.
The Davis--Kahan theorem implies that
\[
\sin \angle(\boldsymbol{\gamma}_1, \hat{\boldsymbol{\gamma}}_1(w))
\le \frac{2 \| \boldsymbol{S}(w) - \E[\boldsymbol{S}(w)] \|_{{\rm op}}}{ \lambda_1(w) - \lambda_2(w) },
\]
where the angle $\angle$ is taken in $[0,\pi/2]$, 
$\| \cdot \|_{{\rm op}}$ denotes the operator norm; see, for example, \cite{RefN08} and \cite{RefV18} for the Davis--Kahan theorem.
Noting that $\left| \boldsymbol{\gamma}_1'\hat{\boldsymbol{\gamma}}_1(w) \right| \le 1$ and $\| \cdot \|_{{\rm op}} \le \| \cdot \|_{\rm F}$, 
we have
\begin{align*}
& \E \left[ \min_{\theta = -1,1} \| \theta \hat{\boldsymbol{\gamma}}_1(w) -  \boldsymbol{\gamma}_1 \|^2 \right] 
= \E \left[ 2 - 2\max_{\theta = -1,1} \left( \theta \boldsymbol{\gamma}_1'\hat{\boldsymbol{\gamma}}_1(w) \right) \right] \\
&\le \E \left[ 2 - 2\left( \boldsymbol{\gamma}_1'\hat{\boldsymbol{\gamma}}_1(w) \right)^2 \right] 
= \E \left[ 2\left( \sin \angle ( \boldsymbol{\gamma}_1, \hat{\boldsymbol{\gamma}}_1(w) ) \right)^2 \right] \\
& \le \E \left[ 2 \frac{4 \| \boldsymbol{S}(w) - \E[\boldsymbol{S}(w)] \|_{{\rm op}}^{2}}{ \left( \lambda_1(w) - \lambda_2(w) \right)^2} \right] \\
& \le \E \left[ \frac{8 \| \boldsymbol{S}(w) - \E[\boldsymbol{S}(w)] \|_{\rm F}^{2}}
{ \left[ \left\{ q(1-w) + (n-1-q)w \right\} (\lambda_1-\lambda_2) + (1-w) \| \boldsymbol{X}\boldsymbol{\alpha} \|^2 \right]^2} \right] \\
&= \frac{8 a \{ q(1 - 2w) + (n-1)w^2 \}  + 16 bc (1-w)^2 }
{ \left[ d \{ q + (n-1-2q)w \} + c (1-w)  \right]^2 },
\end{align*}
which completes the proof.
\end{proof}

\begin{remark}
Although the estimation of $\boldsymbol{\gamma}_1$ is focused on, corresponding results for other eigenvectors of $\boldsymbol{S}(w)$ follows by the almost same method under adequate assumptions on $\boldsymbol{\Sigma}$.
\end{remark}

\begin{remark}\label{remls}
When $\boldsymbol{\alpha} = \boldsymbol{0}_q$, the eigenvector $\hat{\boldsymbol{\gamma}}_1^{(T)} = \hat{\boldsymbol{\gamma}} (1/2)$ of the sample covariance matrix $\boldsymbol{S}_T$ corresponding to the largest eigenvalue satisfies that
\[
\E \left[ \min_{\theta = -1,1} \| \theta \hat{\boldsymbol{\gamma}}_1^{(T)} -  \boldsymbol{\gamma}_1 \|^2 \right] 
\le \frac{8 \{ {\rm tr}(\boldsymbol{\Sigma}^2) + ({\rm tr}(\boldsymbol{\Sigma}))^2 \} }{ (n-1) (\lambda_1 - \lambda_2)^2 }.
\]
\end{remark}

Using Theorem~\ref{thm1}, we will discuss sufficient conditions for the consistency up to sign of $\hat{\boldsymbol{\gamma}}_1(w)$ 
under some asymptotic regimes, as shown in Sections 4 and 5.
As mentioned in Remark~\ref{remc}, when the dimension $p$ is held fixed, the bound of Theorem~\ref{thm1} seems sharp from the viewpoint of consistency.
Here the meaning of the consistency up to sign of $\hat{\boldsymbol{\gamma}}_1(w)$ is that
\begin{equation}\label{conde}
\min_{\theta = -1,1} \| \theta \hat{\boldsymbol{\gamma}}_1(w) - \boldsymbol{\gamma}_1 \| \stackrel{\mathrm{P}}{\to} 0
\end{equation}
for any value of parameters satisfying underlying assumptions, where $\stackrel{\mathrm{P}}{\to}$ denotes convergence in probability.
Note that \eqref{conde} is equivalent to 
\[ \E \left[ \min_{\theta = -1,1} \| \theta \hat{\boldsymbol{\gamma}}_1(w) - \boldsymbol{\gamma}_1 \|^2 \right] \to 0,\]
because 
\[ 0 \leq \min_{\theta = -1,1} \| \theta \hat{\boldsymbol{\gamma}}_1(w) - \boldsymbol{\gamma}_1 \| \leq \sqrt{2} .\]
Whether $\boldsymbol{\alpha} = \boldsymbol{0}_q$ or not is crucial for $\hat{\boldsymbol{\gamma}}_1(w)$ with $w \in [0,1)$, so, in some settings, consistency will be discussed under the assumption that $\boldsymbol{\alpha} \neq \boldsymbol{0}_q$.

\begin{remark}
We use Lemma~\ref{Lem1} to bound $\E[  \| \boldsymbol{S}(w) - \E[\boldsymbol{S}(w)] \|^{2}_{\rm op} ]$ for the purpose of getting a specific bound that will be used to determine the value of $w$ below.
Evaluating $\E[  \| \boldsymbol{S}(w) - \E[\boldsymbol{S}(w)] \|_{\rm op} ]$ directly may provide a better bound.
When $w=1$ or when $\boldsymbol{\alpha} = \boldsymbol{0}_q, w=1/2$, results such as \citet[Theorem 4]{RefKL17} or \citet[Theorem 4.7.1]{RefV18} can be used.
For example, when $\boldsymbol{\alpha} = \boldsymbol{0}_q$, it holds that
\begin{align*}
& \E \left[ \min_{\theta = -1,1} \| \theta \hat{\boldsymbol{\gamma}}_1^{(T)} -  \boldsymbol{\gamma}_1 \|^2 \right] 
\\ &
\le
\E \left[\sqrt{2} \min_{\theta = -1,1} \| \theta \hat{\boldsymbol{\gamma}}_1^{(T)} -  \boldsymbol{\gamma}_1 \| \right] 
\le \frac{ C}{ \lambda_1 - \lambda_2 } \max \left\{ \sqrt{\frac{\lambda_1 {\rm tr}(\boldsymbol{\Sigma}) }{n-1}}, \frac{ {\rm tr} (\boldsymbol{\Sigma}) }{n-1}  \right\} ,
\end{align*}
where $C$ is a positive constant. 
To discuss consistency, this bound is better than the bound in Remark~\ref{remls} especially when $\lambda_1/ {\rm tr}(\boldsymbol{\Sigma})$ is small.
\end{remark}

In practice, we should determine the value of the weight $w$ for the estimator $\hat{\boldsymbol{\gamma}}_1(w)$. 
Other than $w=0,0.5,1$, it may be reasonable to select the value of $w$ minimizing the upper bound, the right-hand side of \eqref{UP}, for the mean squared error.
The value is provided in the following theorem.

\begin{theorem}\label{thm2}
The right-hand side of \eqref{UP} attains the unique minimum at
\begin{equation}\label{w*} 
w^* 
=  \frac{adq + 2bcd}{2adq+2bcd+ac} ,
\end{equation}
where $a$, $b$, $c$, and $d$ are given in \eqref{abcd}.
\end{theorem}

\begin{proof}
The right-hand side of \eqref{UP} is rewritten as
\begin{equation}\label{bouabcd}
8 \frac{ a \{ q(1 - 2w) + (n-1)w^2 \} + 2 bc (1-w)^2 }{ \left[ d \{ q + (n-1-2q)w \} + c (1-w) \right]^2 } ,
\end{equation}
and so the value of $w$ minimizing \eqref{UP} is equivalent to that minimizing
\[ f(w) = \frac{ a \{ q(1 - 2w) + (n-1)w^2 \} + 2 b c (1-w)^2 }{ 2 \left[ d \{ q + (n-1-2q)w \} + c (1-w) \right]^2 }. \]
Since $d \{ q + (n-1-2q)w \} > 0$ for $0 \le w \le 1$ due to $q + (n-1-2q)w =q(1-w)+ (n-1-q)w > 0$ and $d = \lambda_1 - \lambda_2 >0$, 
the function $f(w)$ is continuous on $0 \le w \le 1$. 
Differentiating $f(w)$ with respect to $w$, we have
\begin{align*}
 \frac{\mathrm{d} f (w)}{\mathrm{d} w} 
&=
\frac{ a \{ -2q + 2(n-1)w \} - 4 b c (1-w) }{ 2 \left[ d \{ q + (n-1-2q)w \} + c (1-w) \right]^2 } 
\\ & \qquad
- \frac{  \left[ a \{ q(1 - 2w) + (n-1)w^2 \} + 2 bc (1-w)^2 \right] \left\{ d (n-1-2q) - c \right\} }{ \left[ d \{ q + (n-1-2q)w \} + c (1-w) \right]^3 } \\
&= \frac{ \left\{ (2adq+ac+2bcd)w - (adq + 2bcd) \right\} (n-1-q)}{ \left[ d \{ q + (n-1-2q)w \} + c (1-w) \right]^3 },
\end{align*}
from which we see that $f(w)$ is strictly decreasing on 
\[ 0 \le w<\frac{adq + 2bcd}{2adq+2bcd+ac} \]
and strictly increasing on 
\[ \frac{adq + 2bcd}{2adq+2bcd+ac}<w \le 1. \] 
Hence, the value of $w$ minimizing the function $f(w)$ is given by \eqref{w*}.
This completes the proof.
\end{proof}

A natural estimator for $w^*$ will be given in Section~\ref{sec:6}.
We see from Theorem~\ref{thm2} that 
(i) if $\| \boldsymbol{X} \boldsymbol{\alpha} \|^2 \approx 0$, then $w^* \approx 0.5$, 
(ii) if $\lambda_1-\lambda_2 \approx 0$, then $w^* \approx 0$. 
Moreover, we have the bound $w^* < 2/3$, because it follows from
\[ bd = (\lambda_1 + {\rm tr}(\boldsymbol{\Sigma})) (\lambda_1 - \lambda_2) 
< \lambda_{1}^{2} + \lambda_1 {\rm tr}(\boldsymbol{\Sigma}) 
< {\rm tr}( \boldsymbol{\Sigma}^2 ) + ({\rm tr}(\boldsymbol{\Sigma}))^2 = a \]
that
\[
w^* 
=  \frac{adq + 2bcd}{2adq+2bcd+ac}
<  \frac{ad q + 2 bcd}{2 adq +3 bcd } 
\le \frac{2}{3}.
\]

\section{Consistency under weak identifiability regimes}\label{sec:4}

In this section, we discuss the case of $n \to \infty$ with fixed $p,q$. 
Assume that $n^{-1} \boldsymbol{X}'\boldsymbol{X}$ converges to a positive-definite matrix as $n \to \infty$. 
Let $\boldsymbol{\alpha}$ be fixed. 
Then, $\| \boldsymbol{X} \boldsymbol{\alpha} \|^2 = \Theta (n)$ holds for $\boldsymbol{\alpha} \neq \boldsymbol{0}_q$, 
and  $\| \boldsymbol{X} \boldsymbol{\alpha} \|^2 = 0$ holds for $\boldsymbol{\alpha} = \boldsymbol{0}_q$, 
where $a_n = \Theta (n^u)$ for a sequence $\{ a_n \}_{n \in \mathbb{N}}$ means that 
$\liminf (a_n / n^u) >0$ and $\limsup (a_n / n^u) < \infty$ as $n \to \infty$. 
If $\boldsymbol{\Sigma}$ is fixed (with $\lambda_1-\lambda_2>0$), then $\hat{\boldsymbol{\gamma}}_1(w)$ is always a consistent estimator of $\boldsymbol{\gamma}_1$ up to sign  for all $w \in (0,1]$, because the upper bound for the mean squared error of $\hat{\boldsymbol{\gamma}}_1(w)$ in \eqref{UP} converges to $0$ as follows:
\begin{align*}
& \E \left[ \min_{\theta = -1,1} \| \theta \hat{\boldsymbol{\gamma}}_1(w) -  \boldsymbol{\gamma}_1 \|^2 \right] \nonumber \\
&\le \frac{8 a \{ q(1 - 2w) + (n-1)w^2 \}  + 16 bc (1-w)^2 }{ \left[ d \{ q + (n-1-2q)w \}  + c (1-w)  \right]^2 } \\
&= \frac{8 a \{ q(1 - w)^2 + (n-1-q)w^2 \} + 16 bc (1-w)^2}{ \left[ d \{ q(1-w) + (n-1-q)w \} + c (1-w)  \right]^2 }
= \frac{ \Theta (n)}
{ \left( \Theta (n) \right)^2 } 
= \Theta \left( n^{-1} \right),
\end{align*}
where the notation \eqref{abcd} is used.
When $\boldsymbol{\alpha} \neq \boldsymbol{0}_q$, so is $\hat{\boldsymbol{\gamma}}_1(0)$.
Hence, in this section, we focus on the case of $\lambda_1 - \lambda_2 \to 0$ as $n \to \infty$.
Such regimes are called weak identifiability regimes in \cite{RefPRV20} and \cite{RefPV23} in which some specific structures for $\boldsymbol{\Sigma}$ are assumed. 
These regimes correspond to the cases where the eigenvector of $\boldsymbol{\Sigma}$ corresponding to the largest eigenvalue is not identifiable in the limit.
More specifically, we consider asymptotic regimes $n\to\infty$ under the following assumption.

\begin{assumption}\label{ass1}
As $n \to \infty$ with $q$, $p$, and $\boldsymbol{\alpha}$ held fixed, suppose the followings:
\begin{itemize}
\item[(i)]
$\boldsymbol{X}$ satisfies that $n^{-1} \boldsymbol{X}' \boldsymbol{X}$ converges to a finite positive-definite matrix; 
\item[(ii)]
$\boldsymbol{\Sigma}$ satisfies that
${\rm tr}(\boldsymbol{\Sigma}) = \Theta (1)$, 
${\rm tr}(\boldsymbol{\Sigma}^2) = \Theta (1)$, and
$\lambda_1 - \lambda_2 = \Theta (n^{-\eta})$
for some $\eta > 0$.
\end{itemize}
\end{assumption}

The consistency of $\hat{\boldsymbol{\gamma}}_1(w)$ up to sign is established in the following theorem.

\begin{theorem}\label{thm3}
Let Assumption~\ref{ass1} hold.
Then, as $n\to \infty$ with $q$, $p$, and $\boldsymbol{\alpha}$ held fixed, sufficient conditions for \eqref{conde} are given by
\[ \begin{array}{ll}
\boldsymbol{\alpha} \neq \boldsymbol{0}_q & \mbox{for} \ w=0,\\
\boldsymbol{\alpha} \neq \boldsymbol{0}_q \ \mbox{or} \ \eta< 1/2 & \mbox{for} \ 0<w<1,\\
 \eta< 1/2 & \mbox{for} \ w=1.
\end{array} \]
\end{theorem}

\begin{proof}
In this proof, we use the notation \eqref{abcd} for simplicity.
We notice that $a=\Theta(1)$, $b=\Theta(1)$, $d = \Theta(n^{-\eta})$, and that $c = 0$ under $\boldsymbol{\alpha}=\boldsymbol{0}_q$ and $c = \Theta(n)$ under $\boldsymbol{\alpha} \neq \boldsymbol{0}_q$. 
First, we discuss the case of $\boldsymbol{\alpha} \neq \boldsymbol{0}_q$, that is, $\| \boldsymbol{X} \boldsymbol{\alpha} \|^2 = \Theta (n)$. 
If $0 < w < 1$, then
\begin{align*}
& \E \left[ \min_{\theta = -1,1} \| \theta \hat{\boldsymbol{\gamma}}_1(w) -  \boldsymbol{\gamma}_1 \|^2 \right] \nonumber \\
& \le \frac{8 a \{ q(1 - 2w) + (n-1)w^2 \} + 16 bc (1-w)^2 }
{ \left[ d \{ q + (n-1-2q)w \} + c (1-w) \right]^2 } \\
& = \frac{ \Theta (n) + \Theta (n) }
{ \left( \Theta (n^{1-\eta}) + \Theta (n) \right)^2 } 
 = \frac{ \Theta (n) }
{ \left( \Theta (n) \right)^2 } 
 = \Theta \left( n^{-1} \right).
\end{align*}
If $w=0$, then
\begin{align*}
& \E \left[ \min_{\theta = -1,1} \| \theta \hat{\boldsymbol{\gamma}}_1(0) -  \boldsymbol{\gamma}_1 \|^2 \right] \\
& \le \frac{8 a q  + 16 bc }
{ ( d q + c )^2 } 
= \frac{ \Theta (1) + \Theta (n) }
{ \left( \Theta (n^{-\eta}) + \Theta (n) \right)^2 } 
 = \frac{ \Theta (n) }
{ \left( \Theta (n) \right)^2 } 
 = \Theta \left( n^{-1} \right).
\end{align*}
If $w=1$, then
\begin{align*}
& \E \left[ \min_{\theta = -1,1} \| \theta \hat{\boldsymbol{\gamma}}_1(1) -  \boldsymbol{\gamma}_1 \|^2 \right] \\
& \le \frac{8 a (n-1-q)  }{ d^2 (n-1-q)^2 }
= \frac{ \Theta (n) }{ \left( \Theta (n^{1-\eta})\right)^2 } 
 = \frac{ \Theta (n) }{ \Theta (n^{2-2\eta})  } 
 = \Theta \left( n^{2\eta -1} \right).
\end{align*}
Hence, when $\boldsymbol{\alpha} \neq \boldsymbol{0}_q$, 
\eqref{conde} always holds for $0\le w < 1$ 
but \eqref{conde} holds if $\eta < 1/2$ for $w=1$.

Next, we discuss the case of $\boldsymbol{\alpha} = \boldsymbol{0}_q$. 
If $0 < w \le 1$, then
\begin{align*}
& \E \left[ \min_{\theta = -1,1} \| \theta \hat{\boldsymbol{\gamma}}_1(w) -  \boldsymbol{\gamma}_1 \|^2 \right] \\
& \le \frac{8 a \{ q(1 - 2w) + (n-1)w^2 \} }
{ d^2 \{ q + (n-1-2q)w \}^2 } 
= \frac{ \Theta (n) }{ \left( \Theta (n^{1-\eta}) \right)^2 } 
 = \frac{ \Theta (n) }{\Theta (n^{2-2\eta}) }
 = \Theta \left( n^{2\eta -1} \right).
\end{align*}
If $w=0$, then
\begin{align*}
\E \left[ \min_{\theta = -1,1} \| \theta \hat{\boldsymbol{\gamma}}_1(0) -  \boldsymbol{\gamma}_1 \|^2 \right]
\le \frac{8 a q }{ d^2 q^2 }
= \frac{ \Theta (1) }
{ \left( \Theta (n^{-\eta}) \right)^2 } 
 = \Theta \left( n^{2\eta} \right).
\end{align*}
Hence, when $\boldsymbol{\alpha} = \boldsymbol{0}_q$, \eqref{conde} holds if $\eta < 1/2$ for $0 < w \le 1$ but does not necessarily hold for $w=0$. 

Summarizing the above results, we see that \eqref{conde} with $0<w<1$ holds if at least one of $\boldsymbol{\alpha} \neq \boldsymbol{0}_q$ and $\eta < 1/2$ is satisfied, \eqref{conde} with $w=0$ holds if $\boldsymbol{\alpha} \neq \boldsymbol{0}_q$, and \eqref{conde} with $w=1$ holds if $\eta < 1/2$.
\end{proof}

\begin{remark}\label{remc}
From Remarks~\ref{remloc} and \ref{remls}, when $\boldsymbol{\alpha} = \boldsymbol{0}_q$, if $\lambda_1 - \lambda_2 \to 0$ and $n^{1/2} (\lambda_1 - \lambda_2) \to \infty$, then consistency up to sign of the leading principal eigenvector $\hat{\boldsymbol{\gamma}}_1^{(T)}$ holds.
It is consistent with a result in \citet[Lemma 2.3]{RefPRV20}.
Also from \citet[Lemma 2.3]{RefPRV20}, if $\limsup_{n\to\infty} n^{1/2} (\lambda_1 - \lambda_2) < \infty$ then $\hat{\boldsymbol{\gamma}}_1^{(T)}$ is no more consistent, because their setting is a special case of ours.
\end{remark}

Next, the consistency of $\hat{\boldsymbol{\gamma}}_1(w^*)$ up to sign is established as follows.

\begin{theorem}\label{thm4}
Let Assumption~\ref{ass1} hold.
Then, as $n\to \infty$ with $q$, $p$, and $\boldsymbol{\alpha}$ held fixed, a sufficient condition for 
\begin{equation}\label{condea}
\min_{\theta = -1,1} \| \theta \hat{\boldsymbol{\gamma}}_1(w^*) - \boldsymbol{\gamma}_1 \| \stackrel{\mathrm{P}}{\to} 0
\end{equation}
is given by
\[ \boldsymbol{\alpha} \neq \boldsymbol{0}_q \ \mbox{or} \ \eta<\frac{1}{2}. \]
\end{theorem}

\begin{proof}
In this proof, we use the notation \eqref{abcd} for simplicity, and remark that $w^* = 1/2$ under $\boldsymbol{\alpha}=\boldsymbol{0}_q$ and that $w^* = \Theta(n^{-\eta})$ under $\boldsymbol{\alpha} \neq \boldsymbol{0}_q$. 
For $\boldsymbol{\alpha}=\boldsymbol{0}_q$, since $w^* = 1/2$, we see from Theorem~\ref{thm3} that $\eta < 1/2$ is a sufficient condition for \eqref{condea}. 
For $\boldsymbol{\alpha} \neq \boldsymbol{0}_q$,
\begin{align*}
& \E \left[ \min_{\theta = -1,1} \| \theta \hat{\boldsymbol{\gamma}}_1(w^*) -  \boldsymbol{\gamma}_1 \|^2 \right] \nonumber \\
& \le \frac{8 a \{ q(1 - w^*)^2 + (n-1-q) (w^*)^2 \} + 16 bc (1-w^*)^2 }{ \left[ d \{ q(1-w^*) + (n-1-q)w^* \} + c (1-w^*) \right]^2 } \\
& = \frac{ \Theta (1) + \Theta (n^{1-2\eta}) + \Theta (n)}{ \left( \Theta (n^{-\eta}) + \Theta (n^{1-2\eta}) + \Theta (n) \right)^2 } 
 = \frac{\Theta (n)}
{ \left( \Theta (n) \right)^2 } 
 = \Theta \left( n^{-1} \right),
\end{align*}
which ensures that \eqref{condea} holds.
\end{proof}

\section{Consistency under large $p$, large $n$ regimes}\label{sec:5}

In this section, we consider ``large $p$, large $n$'' regimes under which $n$ and $p$ tend to positive infinity simultaneously.
As introduced in Section~\ref{sec:1}, these regimes are frequently discussed in the literature.
We suppose the following assumption.

\begin{assumption}\label{ass2}
As $n$ and $p$ tend to $\infty$ satisfying that ${n}/{p^{\delta}}$ converges to a finite positive value for some $\delta \in (0,\infty)$ with $q$ held fixed, suppose the follwings:
\begin{itemize}
\item[(i)]
$\boldsymbol{X}$ and $\boldsymbol{\alpha}$ satisfy that ${\| \boldsymbol{X} \boldsymbol{\alpha} \|^2} = \Theta(p^\delta)$;
\item[(ii)]
$\boldsymbol{\Sigma}$ satisfies that ${{\rm tr}(\boldsymbol{\Sigma})} = \Theta(p)$ and $(\lambda_1-\lambda_2) = \Theta(p^\beta)$ for some $\beta \in (-\infty,1]$.
\end{itemize}
\end{assumption}

Under Assumption~\ref{ass2}, $\beta > 0$ implies that the largest eigenvalue is (asymptotically) spiked in the sense that it diverges to positive infinity, while non-spiked eigenvalues imply $\beta \leq 0$.
The cases of $\beta < 0$ correspond to weak identifiability regimes. 
Moreover, $\delta < 1$ implies $p/n \to \infty$.

\begin{remark}
Spiked eigenvalue assumptions on covariance matrices are typically assumed in high-dimensional multivariate analysis; see \cite{RefSSZM16}, \cite{RefWF17}, \cite{RefYA13}, and so on.
\end{remark}

The following theorem provides sufficient conditions for the consistency up to sign of $\hat{\boldsymbol{\gamma}}_1(w)$.
Recall that the meaning of the consistency is explained in \eqref{conde} that works even when the dimension $p$ of the parameter $\boldsymbol{\gamma}_1$ diverges.

\begin{theorem}\label{thm5}
Let Assumption~\ref{ass2} hold.
Then, as $p$ and $n$ tend to $\infty$ satisfying that ${n}/{p^{\delta}}$ converges to a finite positive value  with $q$ held fixed, sufficient conditions for \eqref{conde} are given by
\[ \begin{array}{ll}
\delta > 1 & \mbox{for} \ w=0,\\
\delta > \min \{ 2-2\beta, 2 \} & \mbox{for} \ 0<w<1,\\
\delta > 2-2\beta & \mbox{for} \ w=1 .
\end{array} \]
\end{theorem}

\begin{proof}
In this proof, we use the notation \eqref{abcd} for simplicity and remark that $a = \Theta(p^2)$, $b = \Theta(p)$, $c= \Theta(p^\delta)$, $d = \Theta(p^\beta)$.
For $0 < w < 1$, it holds that
\begin{align*}
& \E \left[ \min_{\theta = -1,1} \| \theta \hat{\boldsymbol{\gamma}}_1(w) -  \boldsymbol{\gamma}_1 \|^2 \right] \\
& \le \frac{8 a \{ q(1 - 2w) + (n-1)w^2 \} + 16 bc (1-w)^2 }
{ \left[ d \{ q + (n-1-2q)w \} + c (1-w) \right]^2 } \\
&= \frac{\Theta(p^{2+\delta}) + \Theta(p^{1+\delta})}{\left( \Theta(p^{\beta+\delta}) + \Theta(p^{\delta})\right)^2} 
= \frac{\Theta(p^{2+\delta})}{\Theta(p^{\max \{ 2\beta + 2\delta, 2\delta \}})}
= \Theta \left( p^{ \min\{ 2 - 2\beta - \delta, 2 - \delta \} } \right),
\end{align*}
which means that if $\delta > \min\{ 2 - 2\beta, 2 \}$, then the mean squared error of $\hat{\boldsymbol{\gamma}}_1(w)$ 
converges to $0$, and thus \eqref{conde} holds.

For $w=1$,
\begin{align*}
& \E \left[ \min_{\theta = -1,1} \| \theta \hat{\boldsymbol{\gamma}}_1(1) -  \boldsymbol{\gamma}_1 \|^2 \right] 
\le \frac{8 a (n-1-q)}{ d^2 ( n-1-q )^2 } 
= \frac{\Theta(p^{2+\delta})}{\Theta(p^{2\beta + 2\delta})}
= \Theta \left( p^{2 - 2\beta - \delta} \right),
\end{align*}
from which we see that if $\delta > 2 - 2\beta$, 
then \eqref{conde} holds.

For $w=0$,
\begin{align*}
& \E \left[ \min_{\theta = -1,1} \| \theta \hat{\boldsymbol{\gamma}}_1(0) -  \boldsymbol{\gamma}_1 \|^2 \right] \\
& \le \frac{8 a q + 16 bc  }{ (d q + c )^2 } 
= \frac{\Theta(p^2) + \Theta(p^{1+\delta})}{\left( \Theta(p^{\beta}) + \Theta(p^{\delta}) \right)^2} 
= \left\{ \begin{array}{ll}
\Theta \left( p^{1-\delta}\right) & (\delta \ge 1), \\
\Theta \left( p^{ \min \{ 2 - 2\beta, 2 - 2\delta \} }\right) & (\delta < 1).
\end{array} \right.
\end{align*}
Since $\min \{ 2 - 2\beta, 2 - 2\delta \} \geq 0$ holds for $\delta < 1$, we see that \eqref{conde} holds for $\delta > 1$. 

This completes the proof.
\end{proof}

Theorem~\ref{thm5} implies that sufficient conditions for the consistency of $\hat{\boldsymbol{\gamma}}_1(w)$ with $0<w<1$ is 
(i) weaker than $\hat{\boldsymbol{\gamma}}_1(1)$ for $\beta < 0$ (weak identifiability regimes) and the same as $\hat{\boldsymbol{\gamma}}_1(1)$ 
otherwise, 
(ii) weaker than $\hat{\boldsymbol{\gamma}}_1(0)$ for $\beta > 0.5$ (strongly spiked eigenvalue regimes) 
but stronger than $\hat{\boldsymbol{\gamma}}_1(0)$ for $\beta < 0.5$ (weakly spiked eigenvalue or weak identifiability regimes). 
We see that if $\beta > 0.5$ (strongly spiked eigenvalue regimes), then the consistency of $\hat{\boldsymbol{\gamma}}_1(w)$ with $0<w<1$ remains valid even for ${p}/{n} \to \infty$ as long as ${p^{2-2\beta}}/{n} \to 0$ holds. 
We also see that, even for $\beta < 0$ (weak identifiability regimes), we have the consistency of $\hat{\boldsymbol{\gamma}}_1(w)$ with $0<w<1$ if ${p^2}/{n} \to 0$.

Next theorem provides a sufficient condition for the consistency up to sign of $\hat{\boldsymbol{\gamma}}_1(w^*)$.

\begin{theorem}\label{thm6}
Let Assumption~\ref{ass2} hold.
Then, as $p$ and $n$ tend to $\infty$ satisfying that ${n}/{p^{\delta}}$ converges to a finite positive value  with $q$ held fixed, a sufficient condition for \eqref{condea} is given by
\[ \delta > \min \{ 2 - 2\beta, 1 \}. \]
\end{theorem}

\begin{proof}
In this proof, we use the notation \eqref{abcd} for simplicity.
Since
\begin{align*}
w^* 
&= \frac{a d q + 2 bcd }{2 ad q  +2 bcd +ac } \\
&= \frac{\Theta(p^{2+\beta})+\Theta_{p}(p^{1+\beta+\delta})}
{\Theta_{p}(p^{2+\beta})+\Theta_{p}(p^{1+\beta+\delta})+\Theta_{p}(p^{2+\delta})}
= \frac{\Theta(p^{\max \{ 2+\beta, 1+\beta+\delta \}})}
{\Theta_{p}(p^{\max \{ 2+\beta, 2+\delta \}})},
\end{align*}
we notice that, 
for $\delta \le 1$, 
$w^* = \Theta(1)$ under $\beta \ge \delta$ and 
$w^* = \Theta(p^{\beta-\delta})$ under $\beta < \delta$. 
For $\delta > 1$, $w^* = \Theta(p^{\beta-1})$. 
From $0\le w^* < 2/3$, it follows that $1-w^* = \Theta(1)$. 
We also note that \eqref{bouabcd} is equal to
\[
\frac{8 a \{ q(1 - w^*)^2 + (n-1-q)(w^*)^2 \} + 16 bc (1-w^*)^2 }
{ \left[ d \{ q(1-w^*) + (n-1-q)w^* \} + c (1-w^*) \right]^2 }.
\]

First, we discuss the case of $\delta > 1$. 
It holds that
\begin{align*}
& \frac{8 a \{ q(1 - w^*)^2 + (n-1-q)(w^*)^2 \} + 16 bc (1-w^*)^2 }
{ \left[ d \{ q(1-w^*) + (n-1-q)w^* \} + c (1-w^*) \right]^2 } \\
&= \frac{ \Theta(p^{\max\{ 0, \delta + 2\beta -2 \} })  \Theta(p^2) + \Theta(p^{\delta +1 }) }
{ \left[ \Theta(p^{\max\{ 0, \delta +\beta -1 \} }) \Theta(p^{\beta}) + \Theta(p^{\delta}) \right]^2 } 
= \frac{ \Theta(p^{\max\{ 2, \delta + 1, \delta + 2\beta  \} }) }
{ \left( \Theta(p^{\max\{ \delta, \beta, \delta +2\beta -1 \} }) \right)^2 } \\
&= \frac{ \Theta(p^{\max\{ \delta + 1, \delta + 2\beta  \} }) }
{ \left( \Theta(p^{\max\{ \delta, \delta +2\beta -1 \} }) \right)^2 }.
\end{align*}
If $\beta > 1/2$, then
\[
\frac{ \Theta(p^{\max\{ \delta + 1, \delta + 2\beta  \} }) }
{ \left( \Theta(p^{\max\{ \delta, \delta +2\beta -1 \} }) \right)^2 } 
= \frac{ \Theta(p^{\delta + 2\beta}) }
{ \left( \Theta(p^{\delta +2\beta -1}) \right)^2 } 
= \frac{ \Theta(p^{\delta + 2\beta}) }
{ \Theta(p^{2\delta +4\beta -2}) }
= \Theta(p^{-\delta - 2\beta +2}),
\]
which implies \eqref{condea} because $\delta > 1$ and $\beta > 1/2$ ensure $-\delta - 2\beta +2 < 0$. 
If $\beta \le 1/2$, then
\[
\frac{ \Theta(p^{\max\{ \delta + 1, \delta + 2\beta  \} }) }
{ \left( \Theta(p^{\max\{ \delta, \delta +2\beta -1 \} }) \right)^2 } 
= \frac{ \Theta(p^{\delta + 1}) }
{ \left( \Theta(p^{\delta }) \right)^2 } 
= \frac{ \Theta(p^{\delta + 1}) }
{ \Theta(p^{2\delta }) } 
= \Theta(p^{-\delta + 1}),
\]
which also implies \eqref{condea}.

Next, we deal with the case of $\delta \le 1$ and $\beta < \delta$.
It holds that
\begin{align*}
& \frac{8 a \{ q(1 - w^*)^2 + (n-1-q)(w^*)^2 \} + 16 bc (1-w^*)^2 }{ \left[ d \{ q(1-w^*) + (n-1-q)w^* \} + c (1-w^*) \right]^2 } \\
&= \frac{\Theta(p^{\max\{ 0, -\delta + 2\beta \} }) \Theta(p^2) + \Theta(p^{\delta + 1}) }
{ \left( \Theta(p^{\max\{ 0, \beta \}}) \Theta(p^{\beta}) + \Theta(p^{\delta}) \right)^2 } 
= \frac{\Theta(p^{\max\{ 2, -\delta + 2\beta + 2\} }) + \Theta(p^{\delta + 1}) }
{ \left( \Theta(p^{\max\{ \beta, 2\beta \}})  + \Theta(p^{\delta}) \right)^2 } \\
&= \frac{\Theta(p^{\max\{ 2, -\delta + 2\beta + 2\} }) }
{ \left( \Theta(p^{\max\{ \delta, 2\beta \}}) \right)^2 }.
\end{align*}
If $\delta > 2\beta$, then
\[
\frac{\Theta(p^{\max\{ 2, -\delta + 2\beta + 2\} }) }{ \left( \Theta(p^{\max\{ \delta, 2\beta \}}) \right)^2 }
= \frac{\Theta(p^{2}) }{ \left( \Theta(p^{\delta}) \right)^2 } 
= \frac{\Theta(p^{2}) }{ \Theta(p^{2\delta}) } 
= \Theta(p^{-2\delta + 2}),
\]
from which we see that \eqref{condea} does not necessarily hold in this case. 
On the other hand, if $\beta < \delta < 2\beta$, then
\[
\frac{\Theta(p^{\max\{ 2, -\delta + 2\beta + 2\} }) }
{ \left( \Theta(p^{\max\{ \delta, 2\beta \}}) \right)^2 } 
= \frac{\Theta(p^{-\delta + 2\beta + 2}) }
{ \left( \Theta(p^{2\beta}) \right)^2 } 
= \frac{\Theta(p^{-\delta + 2\beta + 2}) }
{ \Theta(p^{4\beta}) }
= \Theta(p^{-\delta -2\beta +2}),
\]
which indicates that \eqref{condea} holds if $\delta > 2 - 2\beta$. 

Finally, we discuss the case of $\delta \le 1$ and $\beta \ge \delta$. 
It holds that
\begin{align*}
& \frac{8 a \{ q(1 - w^*)^2 + (n-1-q)(w^*)^2 \} + 16 bc (1-w^*)^2 }{ \left[ d \{ q(1-w^*) + (n-1-q)w^* \} + c (1-w^*) \right]^2 } \\
&= \frac{\Theta(p^\delta) \Theta(p^2) + \Theta(p^{\delta + 1}) }{ \left( \Theta(p^\delta) \Theta(p^{\beta}) + \Theta(p^{\delta}) \right)^2 } 
= \frac{\Theta(p^{\delta + 2}) }{ \left( \Theta(p^{\delta + \beta}) \right)^2 } 
= \frac{\Theta(p^{\delta + 2}) }{ \Theta(p^{2\delta + 2\beta}) }
= \Theta(p^{-\delta - 2\beta + 2}),
\end{align*}
from which we see that \eqref{condea} holds if $\delta > 2 - 2\beta$. 

From what has already been proved, noting that $2\beta > 2 - \delta$ is not satisfied under $\delta \le 1$ and $\delta > 2\beta$ ($2\beta < \delta \le 2- \delta$), we conclude that \eqref{condea} always holds under $\delta > 1$ and that \eqref{condea} holds for $\delta > 2 - 2\beta$ under $\delta \le 1$.
\end{proof}

We see that, as for sufficient conditions for the consistency, $\hat{\boldsymbol{\gamma}}_1(w^*)$ has the good points both of $\hat{\boldsymbol{\gamma}}_1(w)$ with $0<w<1$ and $\hat{\boldsymbol{\gamma}}_1(0)$. 
We think that this may be because the inequality $\| \cdot \|_{\mathrm{op}} \leq \| \cdot \|_{\mathrm{F}}$ is used, so sufficient conditions in Theorem~\ref{thm5} may be relaxed by avoiding this inequality.

\section{Feasible $\hat{\boldsymbol{\gamma}}_1(w^*)$ estimator}\label{sec:6}

Since $w^*$ includes unknown parameters, in order to use $\hat{\boldsymbol{\gamma}}_1(w^*)$, $w^*$ needs to be estimated.
In this section, $n > 2+ q$ is additionally assumed.
Then, a natural estimator of $w^*$ is given by
\[ \hat{w}^* 
= \frac{ \hat{a}\hat{d}q + 2 \hat{b} \hat{c} \hat{d} }{ 2 \hat{a} \hat{d} q +2 \hat{b} \hat{c} \hat{d} + \hat{a} \hat{c}}.
\]
where
\begin{align*}
&\hat{a} = \widehat{{\rm tr}( \boldsymbol{\Sigma}^2 )} + ({\rm tr}(\hat{\boldsymbol{\Sigma}}))^2, \
\hat{b} = \hat{\lambda}_1 + {\rm tr}(\hat{\boldsymbol{\Sigma}}), \
\hat{c} = {\rm tr}( \boldsymbol{S}_R ) - q  {\rm tr}(\hat{\boldsymbol{\Sigma}}), \ 
\hat{d} = \hat{\lambda}_1 - \hat{\lambda}_2, \\
& \hat{\boldsymbol{\Sigma}} = \frac{\boldsymbol{S}_E}{n-1-q}, \
\widehat{{\rm tr}( \boldsymbol{\Sigma}^2 )} = \frac{1}{(n+1-q)(n-2-q)} \left\{ {\rm tr}(\boldsymbol{S}_{E}^{2}) - \frac{\left( {\rm tr}(\boldsymbol{S}_E)\right)^2}{n-1-q}
\right\},  
\end{align*}
and $\hat{\lambda}_1$ and $\hat{\lambda}_2$ are the largest and second largest eigenvalues of $\hat{\boldsymbol{\Sigma}}$, respectively. 
When $\boldsymbol{\alpha} \neq \boldsymbol{0}_q$, noting that
\[ \E[{\rm tr}(\boldsymbol{S}_R)] =   q {\rm tr}( \boldsymbol{\Lambda} ) + \| \boldsymbol{X} \boldsymbol{\alpha} \|^2 , \quad
\V[{\rm tr}(\boldsymbol{S}_R)] =  2q {\rm tr}( \boldsymbol{\Lambda}^2) + 4 \lambda_1 \| \boldsymbol{X} \boldsymbol{\alpha} \|^2 \]
and that
$\boldsymbol{S}_E \sim \mathcal{W}_p(n-1-q, \boldsymbol{\Sigma})$, we can see that $\hat{w}^* - w^* \stackrel{\mathrm{P}}{\to} 0$ under some asymptotic regimes such as a traditional large sample regime where $n \to \infty$ with the assumption that $n^{-1} \boldsymbol{X}' \boldsymbol{X}$ converges to a positive-definite matrix while other parameters are fixed.
Neither the consistency of $\hat{\boldsymbol{\gamma}}_1(w^*)$ nor $\hat{w}^*$ necessarily means that a feasible estimator $\hat{\boldsymbol{\gamma}}_1(\hat{w}^*)$ is consistent up to sign, that is to say,
\begin{equation}\label{condef}
\min_{\theta = -1,1} \| \theta \hat{\boldsymbol{\gamma}}_1(\hat{w}^*) - \boldsymbol{\gamma}_1 \| \stackrel{\mathrm{P}}{\to} 0. 
\end{equation}
Hence, let us discuss asymptotic regimes under which $\hat{w}^* - w^* \stackrel{\mathrm{P}}{\to} 0$ implies \eqref{condef}.

\begin{theorem}\label{thm7}
Under asymptotic regimes satisfying
\begin{align}
& \frac{ \{ q (1 - w^*)^2 + (n-1 -q)(w^*)^2 \} \{ {\rm tr}( \boldsymbol{\Sigma}^2 ) 
+ ({\rm tr}(\boldsymbol{\Sigma}))^2 \} + (\lambda_1 + {\rm tr}(\boldsymbol{\Sigma})) \| \boldsymbol{X} \boldsymbol{\alpha} \|^2 }
{ \left[ \{ q + (n-1-2q)w^* \}(\lambda_1 - \lambda_2) + \| \boldsymbol{X} \boldsymbol{\alpha} \|^2 \right]^2 } \nonumber \\
&\to 0 \label{con2}
\end{align}
and
\begin{align}
& \frac{
(n-1-2q)^2 {\rm tr}(\boldsymbol{\Sigma}^2) + (n-1) \{{\rm tr}(\boldsymbol{\Sigma}^2) + ({\rm tr} (\boldsymbol{\Sigma}) )^2 \}  - (n-2q) \lambda_1  \| \boldsymbol{X} \boldsymbol{\alpha} \|^2 }{ \left[ \{q+(n-1-2q) w^*\}(\lambda_1 - \lambda_2) + \| \boldsymbol{X} \boldsymbol{\alpha} \|^2 \right]^2 }  \nonumber \\
& = O(1) , \label{con1}
\end{align}
if $\hat{w}^* - w^* \stackrel{\mathrm{P}}{\to} 0$, then \eqref{condef} holds.
\end{theorem}

\begin{proof}
In this proof, we use the notation \eqref{abcd} for simplicity.
Since $\boldsymbol{\gamma}_1$ is the first principal eigenvector of $\E[\boldsymbol{S}(w^*)] $ and since
\begin{align*}
\| \boldsymbol{S}(\hat{w}^*) - \E[ \boldsymbol{S}(w)] \|_{\rm op}
& \leq \| \boldsymbol{S}(\hat{w}^*) - \boldsymbol{S}(w^*) \|_{\rm op} + \| \boldsymbol{S}(w^*) - \E[ \boldsymbol{S} (w^*)] \|_{\rm op} \\
& = |\hat{w}^* - w^*| \| \boldsymbol{S}_R - \boldsymbol{S}_E \|_{\rm op} + \| \boldsymbol{S}(w^*) - \E[ \boldsymbol{S} (w^*)] \|_{\rm op},
\end{align*}
we have, by using the Davis--Kahan theorem,
\begin{align*}
\sin  \angle(\boldsymbol{\gamma}_1, \hat{\boldsymbol{\gamma}}_1(\hat{w}^*)) 
& \le \frac{2 \| \boldsymbol{S}(\hat{w}^*) - \E[\boldsymbol{S}(w^*)] \|_{{\rm op}}}{ \lambda_1(w^*) - \lambda_2(w^*) } \\
& \le
2|\hat{w}^* - w^*|  \frac{ \| \boldsymbol{S}_R - \boldsymbol{S}_E \|_{\rm op} }{\lambda_1(w^*) - \lambda_2(w^*)}
+ 2 \frac{\| \boldsymbol{S}(w^*) - \E[ \boldsymbol{S} (w^*)] \|_{\rm op}}{\lambda_1(w^*) - \lambda_2(w^*)} .
\end{align*}
From \eqref{con2}, the second term on the right-hand side converges to 0 in probability.
Thus, it suffices to show that
\[
\frac{ \E[ \| \boldsymbol{S}_R - \boldsymbol{S}_E \|_{\rm op} ] }{\lambda_1(w^*) - \lambda_2(w^*)} = O(1)
\]
because $\hat{w}^* - w^* \stackrel{\mathrm{P}}{\to} 0$ is assumed.
Since the Jensen inequality provides 
\[ \E[\| \boldsymbol{S}_R - \boldsymbol{S}_E \|_{\rm op}] \le \sqrt{\E[\| \boldsymbol{S}_R - \boldsymbol{S}_E \|_{\rm op}^2]} , \] 
we evaluate $\E[ \| \boldsymbol{S}_R - \boldsymbol{S}_E \|^2_{\rm op} ]$.
It follows from
\begin{align*}
& \| \boldsymbol{S}_R - \boldsymbol{S}_E \|_{\mathrm{op}}^2  \\
&\leq \| \boldsymbol{S}_R - \boldsymbol{S}_E \|_{\rm F}^2 
 ={\rm tr} \Bigl( (\boldsymbol{S}_R - \boldsymbol{S}_E)(\boldsymbol{S}_R - \boldsymbol{S}_E)^\top \Bigr) \\
& = 
{\rm tr}\Bigl( (\boldsymbol{S}_R - \E[\boldsymbol{S}_R])^2\Bigr) 
+ {\rm tr}\Bigl( (\boldsymbol{S}_E - \E[\boldsymbol{S}_E])^2\Bigr) 
+  {\rm tr}\Bigl( (\E[\boldsymbol{S}_R - \boldsymbol{S}_E])^2\Bigr) \\ 
& \quad
- 2 {\rm tr}\Bigl( (\boldsymbol{S}_R - \E[\boldsymbol{S}_R]) (\boldsymbol{S}_E - \E[\boldsymbol{S}_E]) \Bigr)
- 2 {\rm tr}\Bigl( (\boldsymbol{S}_R - \E[\boldsymbol{S}_R]) (E[\boldsymbol{S}_R - \boldsymbol{S}_E]) \Bigr)  \\
& \quad
- 2 {\rm tr}\Bigl( (\boldsymbol{S}_E - \E[\boldsymbol{S}_E]) (E[\boldsymbol{S}_R - \boldsymbol{S}_E]) \Bigr),
\end{align*}
that
\begin{align*}
&\E[\| \boldsymbol{S}_R - \boldsymbol{S}_E \|_{\mathrm{op}}^2] \\
&\leq 
\E \left[ {\rm tr}\left( (\boldsymbol{S}_R - \E[\boldsymbol{S}_R])^2 \right)  \right]
+ \E \left[  {\rm tr} \left( (\boldsymbol{S}_E - \E[\boldsymbol{S}_E])^2 \right) \right]
+  {\rm tr} \left( (\E[\boldsymbol{S}_R - \boldsymbol{S}_E])^2 \right). 
\end{align*}
Here, recall that
\[
\E \left[ {\rm tr} \left( \left( \boldsymbol{S}_R - E[\boldsymbol{S}_R] \right)^2 \right) \right] 
=  a q + 2bc\]
and that
\[
\E \left[ {\rm tr} \left( \left( \boldsymbol{S}_E - E[\boldsymbol{S}_E] \right)^2 \right) \right]
= a (n-1-q).
\]
Moreover, it holds that
\begin{align*}
{\rm tr} \left( (\E[\boldsymbol{S}_R - \boldsymbol{S}_E])^2 \right)
&=  {\rm tr} \left( \left\{ - (n-1-2q) \boldsymbol{\Gamma} \boldsymbol{\Lambda} \boldsymbol{\Gamma}' + c \boldsymbol{\gamma}_1 \boldsymbol{\gamma}_1' \right\}^2 \right) \\
&= \left( n-1-2q\right)^2 {\rm tr}(\boldsymbol{\Sigma}^2) + c^2 - 2 c (n-1-2q) \lambda_1 .
\end{align*}
Hence, we have
\begin{align*}
&\E[\| \boldsymbol{S}_R - \boldsymbol{S}_E \|_{\rm op}] \\
&\leq \sqrt{(n-1-2q)^2 {\rm tr}(\boldsymbol{\Sigma}^2) + a (n-1) + c^2 + 2 c \{ b - (n-1-2q) \lambda_1 \} }.
\end{align*}
Noting that 
\[ w^* < 2/3, \ \lambda_1(w^*) - \lambda_2(w^*) 
= d \{q+(n-1-2q) w^*\} + c (1-w^*), \]
and
\begin{align*}
&\frac{(n-1-2q)^2 {\rm tr}(\boldsymbol{\Sigma}^2) + a (n-1) + c^2 + \{ b - (n-1-2q) \lambda_1 \} c }{ \left[ d \{q+(n-1-2q) w^*\} + c (1 - w^*) \right]^2 } \\
&
= \frac{(n-1-2q)^2 {\rm tr}(\boldsymbol{\Sigma}^2) + a (n-1) + \{ b - (n-1-2q) \lambda_1 \}  c}{ \left[ \{q+(n-1-2q) w^*\}d + c (1 - w^*) \right]^2 } + O(1)
\\ 
&
= \frac{(n-1-2q)^2 {\rm tr}(\boldsymbol{\Sigma}^2) + a (n-1) - c (n-2q) \lambda_1 }{ \left[ \{q+(n-1-2q) w^*\}d + c (1 - w^*)  \right]^2 } + o(1) + O(1) ,
\end{align*}
we have the conclusion.
\end{proof}

\begin{remark}
The estimator $\hat{w}^*$ in Theorem~\ref{thm7} can be replaced with arbitrary estimator $\hat{w}$ satisfying that $\hat{w} - w^* \stackrel{\mathrm{P}}{\to} 0$.
\end{remark}

\begin{remark}
The condition \eqref{con1} can be improved by determining the rate of convergence of $\hat{w}^*$:
under asymptotic regimes satisfying \eqref{con2} and
\[
\frac{ (n-1-2q)^2 {\rm tr}(\boldsymbol{\Sigma}^2) + a (n-1)  + c \{ {\rm tr}(\boldsymbol{\Sigma}) - (n  -2q) \lambda_1 \}}{ \left[ d \{q+(n-1-2q) w^*\} + c \right]^2 } 
= o(r_{n}^2)
\]
for some non-decreasing sequence $\{ r_n \}_{n \in \mathbb{N}}$ of positive values $0<r_1 \leq r_2 \leq \cdots$ satisfying that $r_{n} (\hat{w}^* - w^*)$ is uniformly tight, the same conclusion follows.
Here, the notation \eqref{abcd} is used.
\end{remark}

Under the traditional large sample regime, \eqref{con2} and \eqref{con1} are satisfied, so we have the following corollary.

\begin{corollary}
Suppose that $\boldsymbol{\alpha} \neq \boldsymbol{0}_q$ and that $n^{-1} \boldsymbol{X}' \boldsymbol{X}$ converges to a positive-definite matrix as $n\to\infty$ with other parameters held fixed.
Then, as $n\to\infty$, \eqref{condef} holds.
\end{corollary}

The performance of $\hat{\boldsymbol{\gamma}}_1(\hat{w}^*)$, compared with $\hat{\boldsymbol{\gamma}}_1(w^*)$, is examined through numerical simulations in Section~\ref{sec:7}.

\section{Numerical simulations}\label{sec:7}

This section provides the results of some numerical simulations for comparing performances of the estimators $\hat{\boldsymbol{\gamma}}_1^{(T)} =\hat{\boldsymbol{\gamma}}_1 (0.5)$, $\hat{\boldsymbol{\gamma}}_1^{(E)} =\hat{\boldsymbol{\gamma}}_1(1)$, $\hat{\boldsymbol{\gamma}}_1^{(R)} =\hat{\boldsymbol{\gamma}}_1(0)$, $\hat{\boldsymbol{\gamma}}_1(\hat{w}^*)$, and $\hat{\boldsymbol{\gamma}}_1(w^*)$. 
For comparison, $\hat{\boldsymbol{\gamma}}_1(w)$ with $w=0.1,0.2,0.3,0.4,0.6$ are also included. 
The estimators $\hat{\boldsymbol{\gamma}}_1(w)$ with $w=0.7,0.8,0.9$ are not included since $w^* < 2/3$ is shown in Section 3.

First, we consider a traditional large sample regime and a weak identifiability regime, as follows. 
The number of response variables $p$ is 10, the number of explanatory variables $q$ is 5, and 
the sample size $n$ is $20,50,100,200,500$. 
The explanatory variables matrix $\boldsymbol{X}$ ($n \times q$ matrix) is a random sample of size $n$ drawn from 
$\mathcal{N}_q (\boldsymbol{0}_q,\boldsymbol{I}_q)$ with centralization ($\boldsymbol{X}'\boldsymbol{1}_n=\boldsymbol{0}_q$). 
Let $\boldsymbol{\alpha}=\boldsymbol{1}_q$ (the $q\times 1$ vector whose all elements are $1$). 
Let $\boldsymbol{\Sigma} = \boldsymbol{\Gamma}\boldsymbol{\Lambda}\boldsymbol{\Gamma}'$, where we set the diagonal matrix of eigenvalues $\boldsymbol{\Lambda} = {\rm diag} (\lambda_1,\lambda_2,\dots,\lambda_p)$ 
by $\lambda_1=2$ (traditional large sample regime) or $\lambda_1 = 1 + n^{-\eta}$ with $\eta = 1/3,1/2,1$ (weak identifiability regime) and by $\lambda_i=1, \ i=2,3,\dots,p$ and we also set the orthogonal matrix of eigenvectors $\boldsymbol{\Gamma}=(\boldsymbol{\gamma}_1,\boldsymbol{\gamma}_2,\dots,\boldsymbol{\gamma}_p)$ 
by using normalized eigenvectors of a sample covariance matrix of the random sample of size $2p$ drawn from $\mathcal{N}_p (\boldsymbol{0}_p,\boldsymbol{I}_p)$. 
We note that the setting of $\boldsymbol{\Gamma}$ does not affect the results of simulations. 
Table~\ref{tab:01} (traditional large sample regime) and Table~\ref{tab:02} (weak identifiability regime) show the results of mean squared errors of the estimators based on 1000 numerical simulations.

\begin{table}[tbp]
\begin{center}
\caption{
The values of the mean squared errors (with the average of $\hat{w}^*$ and $w^*$) in the case under a traditional large sample regime.}
\label{tab:01}

\begin{tabular}{cccccc}\hline
$n$ & 20 & 50 & 100 & 200 & 500 \\ \hline
$\hat{\boldsymbol{\gamma}}_1^{(T)}$& 0.10517 & 0.03734 & 0.01831 & 0.00862 & 0.00349 \\
$\hat{\boldsymbol{\gamma}}_1^{(E)}$& 0.90481 & 0.46512 & 0.22117 & 0.09505 & 0.03515 \\
$\hat{\boldsymbol{\gamma}}_1^{(R)}$& 0.10723 & 0.03887 & 0.01873 & 0.00886 & 0.00363 \\
$\hat{\boldsymbol{\gamma}}_1(0.1)$ & 0.10404 & 0.03740 & 0.01804 & 0.00852 & 0.00349 \\
$\hat{\boldsymbol{\gamma}}_1(0.2)$ & 0.10113 & 0.03613 & 0.01745 & 0.00823 & 0.00337 \\
$\hat{\boldsymbol{\gamma}}_1(0.3)$ & \textbf{0.09915} & \textbf{0.03531} & \textbf{0.01711} & \textbf{0.00806} & \textbf{0.00330} \\
$\hat{\boldsymbol{\gamma}}_1(0.4)$ & 0.09949 & 0.03540 & 0.01724 & 0.00812 & 0.00331 \\
$\hat{\boldsymbol{\gamma}}_1(0.6)$ & 0.12335 & 0.04309 & 0.02127 & 0.01000 & 0.00400 \\
$\hat{\boldsymbol{\gamma}}_1(\hat{w}^*)$ & 0.10101 & 0.03684 & 0.01782 & 0.00843 & 0.00344 \\
$\hat{w}^*$ & (0.19359) & (0.12519) & (0.11918) & (0.12590) & (0.13702) \\
$\hat{\boldsymbol{\gamma}}_1(w^*)$ & 0.10109 & 0.03642 & 0.01762 & 0.00832 & 0.00341 \\
$w^*$ & (0.19237) & (0.17417) & (0.16823) & (0.16535) & (0.16363) \\ \hline
\end{tabular}

\vspace{2.5truemm}
The best result up to significant digits in each setting is in bold face.
\end{center}
\end{table}

\begin{table}[tbp]
\begin{center}
\caption{The values of the mean squared errors (with the average of $\hat{w}^*$ and $w^*$) in the cases under weak identifiability regimes.}
\label{tab:02}

(a) Results when $\lambda_1 - \lambda_2 = 1/n^{1/3}$.\\
\begin{tabular}{cccccc}\hline
$n$ & 20 & 50 & 100 & 200 & 500 \\ \hline
$\hat{\boldsymbol{\gamma}}_1^{(T)}$& 0.11599 & 0.04301 & 0.02163 & 0.01051 & 0.00420 \\ 
$\hat{\boldsymbol{\gamma}}_1^{(E)}$& 1.25437 & 1.17177 & 1.13561 & 1.06038 & 0.97851 \\ 
$\hat{\boldsymbol{\gamma}}_1^{(R)}$& 0.10449 & 0.03806 & 0.01843 & 0.00890 & 0.00361 \\
$\hat{\boldsymbol{\gamma}}_1(0.1)$ & 0.10343 & 0.03771 & \textbf{0.01836} & \textbf{0.00888} & \textbf{0.00360} \\
$\hat{\boldsymbol{\gamma}}_1(0.2)$ & \textbf{0.10281} & \textbf{0.03759} & 0.01842 & 0.00893 & 0.00361 \\
$\hat{\boldsymbol{\gamma}}_1(0.3)$ & 0.10328 & 0.03796 & 0.01875 & 0.00911 & 0.00367 \\
$\hat{\boldsymbol{\gamma}}_1(0.4)$ & 0.10634 & 0.03936 & 0.01962 & 0.00954 & 0.00383 \\
$\hat{\boldsymbol{\gamma}}_1(0.6)$ & 0.14500 & 0.05212 & 0.02630 & 0.01272 & 0.00504 \\
$\hat{\boldsymbol{\gamma}}_1(\hat{w}^*)$ & 0.10313 & 0.03776 & 0.01838 & 0.00889 & 0.00361 \\
$\hat{w}^*$ & (0.17068) & (0.07585) & (0.04568) & (0.02899) & (0.01705) \\
$\hat{\boldsymbol{\gamma}}_1(w^*)$ & 0.10355 & 0.03784 & 0.01839 & 0.00889 & \textbf{0.00360} \\
$w^*$ & (0.08453) & (0.05606) & (0.04310) & (0.03374) & (0.02475) \\ \hline 
\end{tabular}
\vspace{2.5truemm}

(b) Results when $\lambda_1 - \lambda_2 = 1/n^{1/2}$.\\
\begin{tabular}{cccccc}\hline
$n$ & 20 & 50 & 100 & 200 & 500 \\ \hline
$\hat{\boldsymbol{\gamma}}_1^{(T)}$& 0.12632 & 0.04462 & 0.02185 & 0.01079 & 0.00433 \\
$\hat{\boldsymbol{\gamma}}_1^{(E)}$& 1.33317 & 1.33026 & 1.31682 & 1.32444 & 1.33833 \\
$\hat{\boldsymbol{\gamma}}_1^{(R)}$& 0.10748 & 0.03816 & 0.01868 & 0.00932 & \textbf{0.00363} \\
$\hat{\boldsymbol{\gamma}}_1(0.1)$ & \textbf{0.10697} & \textbf{0.03806} & \textbf{0.01864} & \textbf{0.00929} & \textbf{0.00363} \\
$\hat{\boldsymbol{\gamma}}_1(0.2)$ & 0.10707 & 0.03822 & 0.01872 & 0.00932 & 0.00367 \\ 
$\hat{\boldsymbol{\gamma}}_1(0.3)$ & 0.10854 & 0.03889 & 0.01906 & 0.00948 & 0.00375 \\
$\hat{\boldsymbol{\gamma}}_1(0.4)$ & 0.11343 & 0.04060 & 0.01991 & 0.00987 & 0.00393 \\
$\hat{\boldsymbol{\gamma}}_1(0.6)$ & 0.15839 & 0.05433 & 0.02643 & 0.01296 & 0.00522 \\
$\hat{\boldsymbol{\gamma}}_1(\hat{w}^*)$ & 0.10808 & 0.03809 & 0.01865 & 0.00931 & \textbf{0.00363} \\
$\hat{w}^*$ & (0.17153) & (0.07229) & (0.04350) & (0.02641) & (0.01566) \\
$\hat{\boldsymbol{\gamma}}_1(w^*)$ & 0.10714 & 0.03812 & 0.01866 & 0.00931 & \textbf{0.00363} \\
$w^*$ & (0.05375) & (0.03020) & (0.02054) & (0.01427) & (0.00895) \\ \hline
\end{tabular}
\vspace{2.5truemm}

(c) Results when $\lambda_1 - \lambda_2 = 1/n$.\\
\begin{tabular}{cccccc}\hline
$n$ & 20 & 50 & 100 & 200 & 500 \\ \hline
$\hat{\boldsymbol{\gamma}}_1^{(T)}$& 0.12627 & 0.04431 & 0.02267 & 0.01078 & 0.00442 \\ 
$\hat{\boldsymbol{\gamma}}_1^{(E)}$& 1.44891 & 1.47012 & 1.46049 & 1.47838 & 1.48953 \\ 
$\hat{\boldsymbol{\gamma}}_1^{(R)}$& 0.10464 & \textbf{0.03730} & \textbf{0.01857} & \textbf{0.00905} & \textbf{0.00365} \\ 
$\hat{\boldsymbol{\gamma}}_1(0.1)$ & \textbf{0.10456} & 0.03733 & 0.01865 & 0.00906 & 0.00366 \\
$\hat{\boldsymbol{\gamma}}_1(0.2)$ & 0.10508 & 0.03762 & 0.01888 & 0.00914 & 0.00371 \\ 
$\hat{\boldsymbol{\gamma}}_1(0.3)$ & 0.10699 & 0.03841 & 0.01939 & 0.00934 & 0.00380 \\ 
$\hat{\boldsymbol{\gamma}}_1(0.4)$ & 0.11307 & 0.04022 & 0.02044 & 0.00979 & 0.00400 \\ 
$\hat{\boldsymbol{\gamma}}_1(0.6)$ & 0.16196 & 0.05414 & 0.02773 & 0.01306 & 0.00536 \\ 
$\hat{\boldsymbol{\gamma}}_1(\hat{w}^*)$ & 0.10721 & 0.03733 & 0.01860 & \textbf{0.00905} & \textbf{0.00365} \\ 
$\hat{w}^*$ & (0.17238) & (0.07300) & (0.04313) & (0.02688) & (0.01564) \\ 
$\hat{\boldsymbol{\gamma}}_1(w^*)$ & 0.10461 & \textbf{0.03730} & \textbf{0.01857} & \textbf{0.00905} & \textbf{0.00365} \\ 
$w^*$ & (0.01271) & (0.00440) & (0.00210) & (0.00102) & (0.00040) \\ \hline 
\end{tabular}

\vspace{2.5truemm}
The best result up to significant digits in each setting is in bold face.
\end{center}
\end{table}

We see from Table~\ref{tab:01} that $\hat{\boldsymbol{\gamma}}_1(0.3)$ is the best for all cases and that $\hat{\boldsymbol{\gamma}}_1(\hat{w}^*)$ and $\hat{\boldsymbol{\gamma}}_1(w^*)$ are better than both $\hat{\boldsymbol{\gamma}}_1^{(R)}$ and $\hat{\boldsymbol{\gamma}}_1^{(T)}$. 
These results suggest that the optimal value of $w$, which is unknown in practice, is in $(0,1/2)$ and motivate to use the proposed estimator $\hat{\boldsymbol{\gamma}}_1(\hat{w}^*)$.
As previously anticipated, $\hat{\boldsymbol{\gamma}}_1^{(E)}$ is not better than other estimators.

We see from Table~\ref{tab:02} that $\hat{\boldsymbol{\gamma}}_1(0.1)$ or $\hat{\boldsymbol{\gamma}}_1(0.2)$ is the best when $\eta=1/3$ and that $\hat{\boldsymbol{\gamma}}_1(0.1)$ or $\hat{\boldsymbol{\gamma}}_1^{(R)}$ or $\hat{\boldsymbol{\gamma}}_1(w^*)$ or $\hat{\boldsymbol{\gamma}}_1(\hat{w}^*)$ is the best when $\eta=1/2, 1$.
When $\hat{\boldsymbol{\gamma}}^{(E)}_1$ is consistent $(\eta < 1/2)$, using small but nonzero $w$ tends to improve $\hat{\boldsymbol{\gamma}}_1^{(R)}$ and $\hat{\boldsymbol{\gamma}}_1^{(T)}$.
Particularly, $\hat{\boldsymbol{\gamma}}_1(w^*)$ is slightly better than $\hat{\boldsymbol{\gamma}}_1^{(R)}$.
As for $\hat{\boldsymbol{\gamma}}_1(\hat{w}^*)$, it is not good as $\hat{\boldsymbol{\gamma}}_1(w^*)$ in many cases, but their difference goes to $0$ as $n$ increases. 
In contrast, $\hat{\boldsymbol{\gamma}}_1^{(E)}$ is quite bad especially when $\eta \geq 1/2$, which may be due to the fact from Theorem~\ref{thm3} that $\hat{\boldsymbol{\gamma}}_1^{(E)}$ does not necessarily satisfy \eqref{conde}.
The inconsistency of $\hat{\boldsymbol{\gamma}}_1^{(E)}$ may affect the performance of $\hat{\boldsymbol{\gamma}}_1^{(T)}$, which is worse than $\hat{\boldsymbol{\gamma}}_1^{(R)}$, $\hat{\boldsymbol{\gamma}}_1(\hat{w}^*)$, and $\hat{\boldsymbol{\gamma}}_1(w^*)$. 

Next, we consider large $p$, large $n$ regimes. 
Let $p=20,50,100,200,500$, $q=5$, and $n=\lfloor p^{0.8} \rfloor$, where $\lfloor \cdot \rfloor$ is the floor function.
Let $\boldsymbol{X}$ be a random sample of size $n$ drawn from $\mathcal{N}_q (\boldsymbol{0}_q,\boldsymbol{I}_q)$ with centralization. 
Let $\boldsymbol{\alpha}=\boldsymbol{1}_q$ and $\boldsymbol{\Sigma} = \boldsymbol{\Gamma}\boldsymbol{\Lambda}\boldsymbol{\Gamma}'$, where $\boldsymbol{\Lambda} = {\rm diag} (\lambda_1,\lambda_2,\dots,\lambda_p)$ with $(\lambda_1,\lambda_2) = (p^{0.25},1) , (p^{0.8}, p^{0.4})$, $\lambda_i=1, \ i=3,4,\dots,p$ and $\boldsymbol{\Gamma}=(\boldsymbol{\gamma}_1,\boldsymbol{\gamma}_2,\dots,\boldsymbol{\gamma}_p)$ is set by using normalized eigenvectors of the sample covariance matrix of the random sample of size $2p$ drawn from $\mathcal{N}_p (\boldsymbol{0}_p,\boldsymbol{I}_p)$. 
Table~\ref{tab:03} shows the results of mean squared errors of the estimators based on 1000 numerical simulations.

\begin{table}[tbp]
\begin{center}
\caption{The values of the mean squared errors (with the average of $\hat{w}^*$ and $w^*$) in the cases under large $p$, large $n$ regimes $n = \lfloor p^{0.8} \rfloor$.}
\label{tab:03}

(a) Results when $\lambda_1=p^{0.25}$, $\lambda_2=1$.\\
\begin{tabular}{cccccc}\hline
$p$ & 20 & 50 & 100 & 200 & 500 \\ \hline
$\hat{\boldsymbol{\gamma}}_1^{(T)}$& 0.10647 & 0.07566 & 0.06196 & 0.04910 & 0.03607 \\
$\hat{\boldsymbol{\gamma}}_1^{(E)}$& 0.91101 & 0.43534 & 0.24993 & 0.16127 & 0.09310 \\
$\hat{\boldsymbol{\gamma}}_1^{(R)}$& 0.10856 & 0.08709 & 0.07684 & 0.06605 & 0.05566 \\
$\hat{\boldsymbol{\gamma}}_1(0.1)$ & 0.10488 & 0.08214 & 0.07113 & 0.05983 & 0.04856 \\
$\hat{\boldsymbol{\gamma}}_1(0.2)$ & 0.10163 & 0.07760 & 0.06598 & 0.05436 & 0.04260 \\
$\hat{\boldsymbol{\gamma}}_1(0.3)$ & \textbf{0.09955} & 0.07409 & 0.06199 & 0.05017 & 0.03820 \\
$\hat{\boldsymbol{\gamma}}_1(0.4)$ & 0.10013 & \textbf{0.07276} & \textbf{0.06014} & \textbf{0.04806} & \textbf{0.03584} \\
$\hat{\boldsymbol{\gamma}}_1(0.6)$ & 0.12513 & 0.08655 & 0.06985 & 0.05472 & 0.03944 \\
$\hat{\boldsymbol{\gamma}}_1(\hat{w}^*)$ & 0.10485 & 0.08514 & 0.07521 & 0.06475 & 0.05472 \\
$\hat{w}^*$ & (0.08214) & (0.03544) & (0.02728) & (0.02007) & (0.01245) \\
$\hat{\boldsymbol{\gamma}}_1(w^*)$ & 0.10402 & 0.08336 & 0.07400 & 0.06406 & 0.05440 \\
$w^*$ & (0.12226) & (0.07454) & (0.04898) & (0.03123) & (0.01675) \\ \hline
\end{tabular}
\vspace{2.5truemm}

(b) Results when $\lambda_1=p^{0.8}$, $\lambda_2=p^{0.4}$.\\
\begin{tabular}{cccccc}\hline
$p$ & 20 & 50 & 100 & 200 & 500 \\ \hline
$\hat{\boldsymbol{\gamma}}_1^{(T)}$& \textbf{0.21682} & \textbf{0.10300} & \textbf{0.06570} & \textbf{0.04191} & \textbf{0.02412} \\
$\hat{\boldsymbol{\gamma}}_1^{(E)}$& 0.63574 & 0.16489 & 0.08581 & 0.04842 & 0.02587 \\
$\hat{\boldsymbol{\gamma}}_1^{(R)}$& 0.32229 & 0.29513 & 0.30129 & 0.32524 & 0.35335 \\
$\hat{\boldsymbol{\gamma}}_1(0.1)$ & 0.29480 & 0.21403 & 0.16516 & 0.11736 & 0.06643 \\
$\hat{\boldsymbol{\gamma}}_1(0.2)$ & 0.26881 & 0.15606 & 0.10299 & 0.06423 & 0.03371 \\
$\hat{\boldsymbol{\gamma}}_1(0.3)$ & 0.24459 & 0.12276 & 0.07761 & 0.04807 & 0.02641 \\
$\hat{\boldsymbol{\gamma}}_1(0.4)$ & 0.22383 & 0.10712 & 0.06789 & 0.04296 & 0.02447 \\
$\hat{\boldsymbol{\gamma}}_1(0.6)$ & 0.23027 & 0.10690 & 0.06735 & 0.04251 & 0.02430 \\
$\hat{\boldsymbol{\gamma}}_1(\hat{w}^*)$ & 0.22046 & 0.10730 & 0.06790 & 0.04286 & 0.02451 \\
$\hat{w}^*$ & (0.43745) & (0.43157) & (0.43595) & (0.44012) & (0.43357) \\
$\hat{\boldsymbol{\gamma}}_1(w^*)$ & 0.21935 & 0.10457 & 0.06662 & 0.04241 & 0.02433 \\
$w^*$ & (0.43548) & (0.43582) & (0.43235) & (0.42782) & (0.42058) \\ \hline
\end{tabular}

\vspace{2.5truemm}
The best result up to significant digits in each setting is in bold face.
\end{center}
\end{table}

As for the two cases in Table~\ref{tab:03}, features of the results are different.
We see from Table~\ref{tab:03}-(a) that $\hat{\boldsymbol{\gamma}}_1(0.3)$ is the best for $p=20$ and $\hat{\boldsymbol{\gamma}}_1(0.4)$ is the best for the other cases. 
Performances of $\hat{\boldsymbol{\gamma}}_1(\hat{w}^*)$ and $\hat{\boldsymbol{\gamma}}_1(w^*)$ are not good but better than $\hat{\boldsymbol{\gamma}}_1^{(R)}$.
This result may be because the upper bound in \eqref{UP} is not sharp in high-dimensional and weakly spiked (or non-asymptotically spiked) eigenvalue situations.
We see from Table~\ref{tab:03}-(b) that $\hat{\boldsymbol{\gamma}}_1^{(T)}$ is the best for all cases. 
Although $\hat{\boldsymbol{\gamma}}_1(\hat{w}^*)$ and $\hat{\boldsymbol{\gamma}}_1(w^*)$ are not the best, their performances become closer to $\hat{\boldsymbol{\gamma}}_1^{(T)}$ as $p$ and $n$ increase. 
As for $\hat{\boldsymbol{\gamma}}_1^{(E)}$, its performance is bad for small $p$ and rapidly closer to $\hat{\boldsymbol{\gamma}}_1^{(T)}$ as $p$ and $n$ increase, which may be due to the fact from Theorem~\ref{thm5} that 
$\hat{\boldsymbol{\gamma}}_1^{(E)}$ is consistent because $\delta = 0.8$ and $\beta = 0.8$ satisfy $\delta > 2- 2\beta$. 
In contrast, $\hat{\boldsymbol{\gamma}}_1^{(R)}$ is bad especially for larger $p$, which may be due to the fact from Theorem~\ref{thm5} that $\hat{\boldsymbol{\gamma}}_1^{(R)}$ is not necessarily consistent because $\delta = 0.8$ does not satisfy $\delta > 1$. 


\section{Real data example}\label{sec:8}
In this section, we apply the conventional methods and the proposed method to human dimensions data of \cite{RefKM05}.
In the dataset, there are 521 individuals consisting of 267 males and 254 females with 251 dimensions (variables) in common. 
We select $p=10$ dimensions associated with feet as response variables, and $q=5,10$ dimensions associated with hands as explanatory variables; see Table~\ref{tab:2}.
The cases $q=2$-A and $q=2$-B will be explained later.
In the following, we use the measurements of $n=66$ males of 19--22 years old that have no missing data on the selected dimensions. 

\begin{table}[tbp]
\begin{center}
\begin{footnotesize}
\caption{
The dimensions selected in real data example. \\
}
\label{tab:2}

(a) Response variables.\\
\begin{tabular}{l}\hline
Bimalleolar breadth \\
Medial malleolus height \\
Sphyrion height \\
Lateral malleolus height \\
Sphyrion fibulare height \\
Dorsal arch height \\
Ball height \\
Outside ball height \\
Great toe tip height \\
Foot circumference \\
\hline
\end{tabular}
\vspace{2.5truemm}

(b) Explanatory variables.
The sign ``$\circ$'' denotes the indicated variable is used. \\
\begin{tabular}{lcccc}\hline
& $q=5$ & $q=10$ & $q=2$-A & $q=2$-B \\ \hline
Hand length, tip to wrist crease & $\circ$ & $\circ$ & $\circ$ & -- \\
Palm length & $\circ$ & $\circ$ & -- & -- \\
Thumb length & $\circ$ & $\circ$ & -- &  $\circ$ \\
Index finger length & $\circ$ & $\circ$ & -- &  $\circ$ \\
Middle finger length & -- & $\circ$ & -- & -- \\
2nd and 3rd phalanx length, middle finger & -- & $\circ$ & -- & -- \\
Middle finger length, dorsal & -- & $\circ$ & -- & -- \\
Bicondylar humerus & -- & $\circ$ & -- & -- \\
Bistyloid breadth & -- & $\circ$ & -- & -- \\
Hand breadth, diagonal & $\circ$ & $\circ$ & $\circ$ & -- \\
\hline
\end{tabular}

\end{footnotesize}
\end{center}
\end{table}

\begin{remark}
In the case $q=10$, there is an approximate linear relationship
\[
\mbox{Hand length, tip to wrist crease}
\approx
\mbox{Palm length} + \mbox{Middle finger length},
\]
where the correlation coefficient between the left-hand side and the right-hand side is about 0.96.
This relationship does not mess up the result because the methods to estimate $\boldsymbol{\gamma}_1$ discussed in this paper are based on the projections $\boldsymbol{X} (\boldsymbol{X}' \boldsymbol{X})^{-1} \boldsymbol{X}' \boldsymbol{Y}$ and $(\boldsymbol{I}_n - n^{-1} \boldsymbol{1}_n \boldsymbol{1}_n') \boldsymbol{Y}$ of $\boldsymbol{Y}$.
\end{remark}

Estimates of $\boldsymbol{\gamma}_1$ is provided in Table~\ref{tab:3}.
Note that the values of $\hat{\boldsymbol{\gamma}_1}^{(T)}$ are the same in the two tables because explanatory variables are not used.
The all coefficients in the all estimates calculated here are positive, which suggest that the first principal component represents a size factor.
The contribution ratios of the first two eigenvalues are $0.393, 0.264$ for $q=5$ and $0.405, 0.262$ for $q=10$, so the identifiability of the first principal eigenvector does not seem weak.
The values of $\hat{w}^*$ are $0.230$ for the case $q=5$ and $0.256$ for the case $q=10$, hence the value of $\hat{\boldsymbol{\gamma}}_1(\hat{w}^*)$ is different from the values of $\hat{\boldsymbol{\gamma}}_1^{(R)}$, $\hat{\boldsymbol{\gamma}}_1^{(E)}$, and $\hat{\boldsymbol{\gamma}}_1^{(T)}$.

\begin{table}[tbp]
\begin{center}
\begin{footnotesize}
\caption{
Estimates for $\gamma_1$ in the real data example. \\
}
\label{tab:3}

(a) Results when $q=5$.
The value of $\hat{w}^*$ is 0.230.\\
\begin{tabular}{c|c|c|c|c|c|c|c|c}\hline
$\hat{\boldsymbol{\gamma}}_1^{(T)}$ & $\hat{\boldsymbol{\gamma}}_1^{(E)}$ & 
$\hat{\boldsymbol{\gamma}}_1^{(R)}$ & $\hat{\boldsymbol{\gamma}}_1(0.1)$ & $\hat{\boldsymbol{\gamma}}_1(0.2)$ & 
$\hat{\boldsymbol{\gamma}}_1(0.3)$ & $\hat{\boldsymbol{\gamma}}_1(0.4)$ & $\hat{\boldsymbol{\gamma}}_1(0.6)$ & 
$\hat{\boldsymbol{\gamma}}_1(\hat{w}^*)$  \\ \hline
0.181& 	0.159& 	0.188& 	0.187& 	0.186& 	0.184& 	0.183& 	0.178& 	0.185\\
0.329& 	0.371& 	0.319& 	0.321& 	0.322& 	0.324& 	0.326& 	0.333& 	0.323\\
0.305& 	0.371& 	0.288& 	0.290& 	0.293& 	0.296& 	0.300& 	0.312& 	0.294\\
0.245& 	0.209& 	0.247& 	0.248& 	0.247& 	0.247& 	0.247& 	0.244& 	0.247\\
0.165& 	0.044& 	0.206& 	0.201& 	0.194& 	0.186& 	0.176& 	0.151& 	0.191\\
0.202& 	0.336& 	0.150& 	0.157& 	0.166& 	0.176& 	0.188& 	0.218& 	0.169\\
0.076& 	0.101& 	0.065& 	0.066& 	0.068& 	0.070& 	0.073& 	0.079& 	0.069\\
0.034& 	0.028& 	0.036& 	0.036& 	0.035& 	0.035& 	0.035& 	0.033& 	0.035\\
0.058& 	0.072& 	0.052& 	0.053& 	0.054& 	0.055& 	0.056& 	0.059& 	0.054\\
0.792& 	0.725& 	0.803& 	0.802& 	0.801& 	0.799& 	0.796& 	0.787& 	0.800\\ \hline
\end{tabular}
\vspace{2.5truemm}

(b) Results when $q=10$.
The value of $\hat{w}^*$ is 0.256.\\
\begin{tabular}{c|c|c|c|c|c|c|c|c}\hline
$\hat{\boldsymbol{\gamma}}_1^{(T)}$ & $\hat{\boldsymbol{\gamma}}_1^{(E)}$ & 
$\hat{\boldsymbol{\gamma}}_1^{(R)}$ & $\hat{\boldsymbol{\gamma}}_1(0.1)$ & $\hat{\boldsymbol{\gamma}}_1(0.2)$ & 
$\hat{\boldsymbol{\gamma}}_1(0.3)$ & $\hat{\boldsymbol{\gamma}}_1(0.4)$ & $\hat{\boldsymbol{\gamma}}_1(0.6)$ & 
$\hat{\boldsymbol{\gamma}}_1(\hat{w}^*)$ \\ \hline
0.181& 	0.150& 	0.191& 	0.189& 	0.188& 	0.186& 	0.183& 	0.177& 	0.187\\
0.329& 	0.393& 	0.312& 	0.315& 	0.317& 	0.320& 	0.324& 	0.335& 	0.319\\
0.305& 	0.394& 	0.283& 	0.286& 	0.289& 	0.293& 	0.299& 	0.314& 	0.291\\
0.245& 	0.215& 	0.247& 	0.247& 	0.247& 	0.247& 	0.246& 	0.244& 	0.247\\
0.165& 	0.051& 	0.204& 	0.199& 	0.192& 	0.185& 	0.176& 	0.152& 	0.188\\
0.202& 	0.335& 	0.153& 	0.160& 	0.168& 	0.177& 	0.188& 	0.218& 	0.173\\
0.076& 	0.096& 	0.067& 	0.068& 	0.070& 	0.071& 	0.073& 	0.078& 	0.071\\
0.034& 	0.020& 	0.040& 	0.039& 	0.038& 	0.037& 	0.036& 	0.032& 	0.037\\
0.058& 	0.070& 	0.053& 	0.054& 	0.054& 	0.055& 	0.056& 	0.059& 	0.055\\
0.792& 	0.701& 	0.807& 	0.805& 	0.803& 	0.800& 	0.797& 	0.786& 	0.802\\ \hline
\end{tabular}
\end{footnotesize}
\end{center}
\end{table}

In addition, we conducted the cross-validation in order to compare the proposed estimator $\hat{\boldsymbol{\gamma}}_1(\hat{w}^*)$ and conventional estimators from the viewpoint of prediction, which is not directly discussed in this paper.
The prediction based on ordinary least squares (OLS) estimator is also considered for reference.
Particularly, we calculate the values of the cross-validated mean squared prediction error
\[
\frac{1}{66} \sum_{i=1}^{66} \| \boldsymbol{y}_i - \hat{\boldsymbol{y}}_i \|^2 ,
\]
where $\hat{\boldsymbol{y}}_i $ is a predictor for $\boldsymbol{y}_i$, the $i$-th row of $\boldsymbol{Y}$, constructed from the sample without $\boldsymbol{y}_i$ for $i=1,\ldots,66$.
The result is provided in Table~\ref{tab:4}.
This table shows that from the viewpoint of prediction, $\hat{\boldsymbol{\gamma}}_1(\hat{w}^*)$ is slightly better than the conventional estimators for the data analyzed here, while the result depends on data.
Finally, two additional results with $q=2$ explanatory variables are provided; see Table~\ref{tab:2}-(b) for the dimensions selected.
The cases $q=2$-A and $q=2$-B are results when good and bad sets of explanatory variables are considered, respectively.
The values of $\hat{w}^*$ are $0.221$ and $0.468$, respectively, which implicitly suggest that the regression relatively works for the case $q=2$-A and does not work for $q=2$-B.
As for the case $q=2$-A, $\hat{\boldsymbol{\gamma}}_1(\hat{w}^*)$ is also better than $\hat{\boldsymbol{\gamma}}_1^{(T)}$, $\hat{\boldsymbol{\gamma}}_1^{(E)}$, and $\hat{\boldsymbol{\gamma}}_1^{(R)}$, and OLS is the best.
This result seems natural because only good explanatory variables are used, so variances becomes smaller.
As for the case $q=2$-B, $\hat{\boldsymbol{\gamma}}_1^{(E)}$ is the best, $\hat{\boldsymbol{\gamma}}_1(\hat{w}^*)$ is better than $\hat{\boldsymbol{\gamma}}_1^{(R)}$, and $\hat{\boldsymbol{\gamma}}_1^{(R)}$ and OLS are bad.

\begin{table}[tbp]
\begin{center}
\caption{The values of the cross-validated mean squared prediction error.}
\label{tab:4}

\begin{tabular}{ccccc}\hline
$q$ & 5 & 10 & 2-A & 2-B \\ \hline
$\hat{\boldsymbol{\gamma}}_1^{(T)}$& 189.288 & 200.642 & 188.771 & 272.478 \\ 
$\hat{\boldsymbol{\gamma}}_1^{(E)}$& 196.159 & 206.911 & 198.906 & \textbf{271.936} \\
$\hat{\boldsymbol{\gamma}}_1^{(R)}$& 189.282 & 200.815 & 188.542 & 274.767 \\
$\hat{\boldsymbol{\gamma}}_1(0.1)$ & 189.197 & 200.714 & 188.460 & 273.879 \\
$\hat{\boldsymbol{\gamma}}_1(0.2)$ & 189.136 & 200.629 & 188.416 & 273.324 \\
$\hat{\boldsymbol{\gamma}}_1(0.3)$ & 189.112 & 200.573 & 188.428 & 272.949 \\
$\hat{\boldsymbol{\gamma}}_1(0.4)$ & 189.150 & \textbf{200.565} & 188.529 & 272.680 \\
$\hat{\boldsymbol{\gamma}}_1(0.6)$ & 189.589 & 200.861 & 189.235 & 272.321 \\
$\hat{\boldsymbol{\gamma}}_1(\hat{w}^*)$ & \textbf{189.111} & 200.573 & 188.411 & 272.510 \\
OLS & 189.991 & 214.441 & \textbf{186.147} & 274.463 \\ \hline
\end{tabular}

\vspace{2.5truemm}
The smallest value up to significant digits in each setting is in bold face.
\end{center}
\end{table}

\section{Concluding remarks}\label{sec:9}
In the multivariate allometric regression model \eqref{MAR}, we considered the problem of estimating $\boldsymbol{\gamma}_1$, a key quantity for inference of this model.
Specifically, we considered a class of estimators $\hat{\boldsymbol{\gamma}}_1(w)$ $(0 \leq w \leq 1)$ based on the weighted sum-of-squares matrices $\boldsymbol{S}_R$ and $\boldsymbol{S}_E$, established a non-asymptotic upper bound of their mean squared errors, derived the value $w^*$ that attains the minimum of the upper bound with respect to $w$, proposed its estimator $\hat{w}^*$, and discussed sufficient conditions for the consistency up to sign of  $\hat{\boldsymbol{\gamma}}_1(w)$ $(0 \leq w \leq 1)$, $\hat{\boldsymbol{\gamma}}_1(w^*)$, and $\hat{\boldsymbol{\gamma}}_1(\hat{w}^*)$.
Based on the result of numerical simulations and a real data example, we think that $\hat{\boldsymbol{\gamma}}_1(\hat{w}^*)$ and $\hat{\gamma}_1^{(T)} = \hat{\boldsymbol{\gamma}}_1(0.5)$ are good from the viewpoint of estimation of $\boldsymbol{\gamma}_1$ and prediction when regression works.
Due to a weight $w$ other than $1/2$,  $\hat{\boldsymbol{\gamma}}_1(\hat{w}^*)$ is more flexible than $\hat{\gamma}_1^{(T)}$, which seems a favorable property because $w=1/2$ imposes a heavy weight on $\boldsymbol{S}_E$ when $\| \boldsymbol{X} \boldsymbol{\alpha} \|^2$ is not small and when $\lambda_1 - \lambda_2$ is small, while $\hat{\boldsymbol{\gamma}}_1(\hat{w}^*)$ does not work well when $\lambda_1-\lambda_2$ is large but $\lambda_1/{\rm tr}(\boldsymbol{\Sigma})$ is small.

\section*{Acknowledgments}
This study was supported in part by Japan Society for the Promotion of Science KAKENHI Grant Numbers 20K11713 and 21K13836.
The authors would like to thank National Institute of Advanced Industrial Science and Technology (AIST) for providing the data used in Section~\ref{sec:8}.

\end{document}